\documentclass[a4paper,leqno,12pt]{amsart}
\usepackage{amsmath,amsthm}
\usepackage{amssymb}
\usepackage{xspace}
\usepackage{bm}
\usepackage{amstext}
\usepackage{amsfonts}
\usepackage{graphicx}
\usepackage[mathscr]{euscript}
\usepackage{amscd}
\usepackage{latexsym}
\usepackage{amssymb}
\usepackage{enumerate}
\usepackage{mathrsfs}
\usepackage[cmtip,all]{xy}
\usepackage{xcolor}
\usepackage{standalone}
\usepackage{tikz}
\usetikzlibrary{arrows.meta}
\usepackage{tikz-cd}
\usetikzlibrary{intersections}
\usetikzlibrary{calc}
\usetikzlibrary{decorations.markings}
\usetikzlibrary{arrows}
\usetikzlibrary{shapes,arrows,shapes.multipart}
\theoremstyle{plain}
\swapnumbers
    \newtheorem{thm}[figure]{Theorem}%[section]
    \newtheorem{prop}[figure]{Proposition}
    \newtheorem{lemma}[figure]{Lemma}
    
    \newtheorem{corollary}[figure]{Corollary}
        
    \newtheorem{subsec}[figure]{}
    
    \newtheorem*{thma}{Theorem A}

\theoremstyle{definition}
    \newtheorem{defn}[figure]{Definition}
    
    \newtheorem{notn}[figure]{Notation}

\theoremstyle{remark}
        \newtheorem{remark}[figure]{Remark}
        
        \newtheorem{example}[figure]{Example}

\renewcommand{\thefigure}{\arabic{section}.\arabic{figure}}
\newenvironment{mysubsection}[2][]
{\begin{subsec}\begin{upshape}\begin{bfseries}{#2.}
\end{bfseries}{#1}}
{\end{upshape}\end{subsec}}
\newcommand{\sect}{\setcounter{figure}{0}\section}
\newcommand{\hs}{\hspace*{5 mm}}

\newcommand{\hsm}{\hspace*{2 mm}}
\newcommand{\vsn}{\vspace{2 mm}}

\newcommand{\vsm}{\vspace{3 mm}}

\newcommand{\wh}{\ -- \ }
\newcommand{\wwh}{-- \ }
\newcommand{\w}[2][ ]{\ \ensuremath{#2}{#1}\ }
\newcommand{\ww}[1]{\ \ensuremath{#1}}

\newcommand{\wwb}[1]{\ \ensuremath{(#1)}-}
\newcommand{\wb}[2][ ]{\ (\ensuremath{#2}){#1}\ }
\newcommand{\wref}[2][ ]{\ (\ref{#2}){#1}\ }

\newcommand{\vin}{\rotatebox[origin=c]{-90}{$\in$}}
\newenvironment{psmallmatrix}
  {\left(\begin{smallmatrix}}
  {\end{smallmatrix}\right)}
\newenvironment{myeq}[1][]
{\stepcounter{figure}\begin{equation}\tag{\thefigure}{#1}}
{\end{equation}}

\newcommand{\mydiagram}[2][]
{\stepcounter{figure}\begin{equation}
     \tag{\thefigure}{#1}\vcenter{\xymatrix{#2}}\end{equation}}
\newcommand{\myodiag}[2][]
{\stepcounter{figure}\begin{equation}
     \tag{\thefigure}{#1}\vcenter{\xymatrix@R=10pt@C=1pt{#2}}\end{equation}}
\newcommand{\mypdiag}[2][]
{\stepcounter{figure}\begin{equation}
     \tag{\thefigure}{#1}\vcenter{\xymatrix@R=-1pt@C=5pt{#2}}\end{equation}}
\newcommand{\myqdiag}[2][]
{\stepcounter{figure}\begin{equation}
     \tag{\thefigure}{#1}\vcenter{\xymatrix@R=5pt@C=10pt{#2}}\end{equation}}
\newcommand{\myrdiag}[2][]
{\stepcounter{figure}\begin{equation}
     \tag{\thefigure}{#1}\vcenter{\xymatrix@R=3pt@C=25pt{#2}}\end{equation}}
\newcommand{\mysdiag}[2][]
{\stepcounter{figure}\begin{equation}
     \tag{\thefigure}{#1}\vcenter{\xymatrix@R=0pt@C=20pt{#2}}\end{equation}}
\newcommand{\mytdiag}[2][]
{\stepcounter{figure}\begin{equation}
     \tag{\thefigure}{#1}\vcenter{\xymatrix@R=0pt@C=-27pt{#2}}\end{equation}}
\newcommand{\myudiag}[2][]
{\stepcounter{figure}\begin{equation}
     \tag{\thefigure}{#1}\vcenter{\xymatrix@R=9pt@C=20pt{#2}}\end{equation}}
\newcommand{\myvdiag}[2][]
{\stepcounter{figure}\begin{equation}
     \tag{\thefigure}{#1}\vcenter{\xymatrix@R=16pt@C=36pt{#2}}\end{equation}}
\newcommand{\mywdiag}[2][]
{\stepcounter{figure}\begin{equation}
     \tag{\thefigure}{#1}\vcenter{\xymatrix@R=15pt@C=26pt{#2}}\end{equation}}
\newcommand{\myxdiag}[2][]
{\stepcounter{figure}\begin{equation}
     \tag{\thefigure}{#1}\vcenter{\xymatrix@R=10pt@C=2pt{#2}}\end{equation}}
\newcommand{\myydiag}[2][]
{\stepcounter{figure}\begin{equation}
     \tag{\thefigure}{#1}\vcenter{\xymatrix@R=10pt@C=12pt{#2}}\end{equation}}
\newcommand{\hra}{\hookrightarrow}

\newcommand{\xra}[1]{\xrightarrow{#1}}

\newcommand{\epic}{\to\hspace{-5 mm}\to}

\newcommand{\lra}[1]{\langle{#1}\rangle}
\newcommand{\up}[1]{\sp{(#1)}}
\newcommand{\bup}[1]{\sp{[{#1}]}}
\newcommand{\lo}[1]{\sb{(#1)}}

\newcommand{\rest}[1]{\lvert\sb{#1}}
\newcommand{\Cat}{\mbox{\sf Cat}}

\newcommand{\Ch}{\mbox{\sf Ch}}
\newcommand{\Chp}{\Ch\sb{+}}
\newcommand{\Path}{\mbox{\sf Incomp}}
\newcommand{\wPath}{\widehat{\Path}}

\newcommand{\Vect}{\mbox{\sf Vect}}
\newcommand{\grV}{\gr\Vect}
\newcommand{\sI}{\mathscr{I}}

\newcommand{\sIu}[1]{\sI\up{#1}}

\newcommand{\hID}{\sI\sb{\Dr}}
\newcommand{\hIDu}[1]{\hID\up{#1}}
\newcommand{\hII}{\sI\sb{\Ind}}
\newcommand{\hIIn}[1]{\sI\bup{#1}\sb{\Ind}}

\newcommand{\sJ}{\mathscr{J}}

\newcommand{\hJD}{\sJ\sb{\Dr}}

\newcommand{\hJI}{\sJ\sb{\Ind}}
\newcommand{\hJIn}[1]{\hJI\bup{#1}}

\newcommand{\hJH}{\sJ\sb{\Hybr}}
\newcommand{\tJ}{\widetilde{\sJ}}
\newcommand{\tJD}{\tJ\sb{\Dr}}

\newcommand{\aci}[1]{{#1}/\sI}
\newcommand{\cib}[1]{\sI/{#1}}
\newcommand{\cibb}[1]{\sI/\!/{#1}}
\newcommand{\acb}[3]{{#1}/{#2}/{#3}}
\newcommand{\aib}[2]{\acb{#1}{\sI}{#2}}
\newcommand{\acbb}[3]{{#1}/{#2}/\!/{#3}}
\newcommand{\aibb}[2]{\acbb{#1}{\sI}{#2}}
\newcommand{\acj}[1]{{#1}/\sJ}

\newcommand{\ajb}[2]{\acb{#1}{\sJ}{#2}}

\newcommand{\bF}{\mathbb{F}}
\newcommand{\tS}{\widetilde{\Sigma}}
\newcommand{\sh}[1]{\tS\sp{#1}}

\newcommand{\Cok}{\operatorname{Cok}}
\newcommand{\colim}{\operatornamewithlimits{colim}}
\newcommand{\Dr}{\operatorname{Der}}
\newcommand{\Der}[1]{\Dr\up{#1}}
\newcommand{\Exp}{\operatorname{Exp}}
\newcommand{\Fib}{\operatorname{Fib}}
\newcommand{\Fil}{\operatorname{Fil}}
\newcommand{\gr}{\operatorname{gr}}
\newcommand{\hocolim }{\operatornamewithlimits{hocolim}}
\newcommand{\ho}{\operatorname{ho}}
\newcommand{\Hom}{\operatorname{Hom}}
\newcommand{\Hybr}{\operatorname{Hyb}}
\newcommand{\Hyb}[1]{\Hybr\up{#1}}
\newcommand{\oHyb}[1]{\overline{\Hybr}\sb{#1}}
\newcommand{\Id}{\operatorname{Id}}

\newcommand{\inc}{\operatorname{inc}}
\newcommand{\Incomp}{\operatorname{Inc}}
\newcommand{\Ind}{\operatorname{Ind}}
\newcommand{\init}{\operatorname{init}}
\newcommand{\Ker}{\operatorname{Ker}}
\newcommand{\Obj}{\operatorname{Obj}}
\newcommand{\Proj}{\operatorname{Proj}}

\newcommand{\term}{\operatorname{term}}
\newcommand{\fD}{\mathfrak{D}}

\newcommand{\fM}[1]{\mathfrak{M}\up{#1}}

\newcommand{\fX}{\mathfrak{X}}
\newcommand{\tX}{\widetilde{\fX}}

\newcommand{\wXn}[1]{\fX\bup{#1}}
\newcommand{\wXi}{\wXn{\infty}}
\newcommand{\fY}{\mathfrak{Y}}

\newcommand{\tY}{\widetilde{\fY}}
\newcommand{\ttY}{\widetilde{\fY'}}
\newcommand{\wY}{\widehat{\fY}}
\newcommand{\wYu}{\tY\bup{2}}
\newcommand{\wYn}[1]{\fY\bup{#1}}

\newcommand{\fZ}{\mathfrak{Z}}

\newcommand{\dt}[1]{d\sp{2}\lo{#1}}
\newcommand{\hdtw}{\hat{d}\sp{2}}
\newcommand{\hdt}[1]{\hdtw\lo{#1}}
\newcommand{\cC}{\mathcal{C}}
\newcommand{\eK}{\mathcal{K}}

\newcommand{\Css}{C\sb{\ast\ast}}
\newcommand{\pd}[1]{\partial\sb{#1}}
\newcommand{\tpd}[1]{\widetilde{\partial}\sb{#1}}
\newcommand{\Hs}{H\sb{\ast}}
\newcommand{\DD}[4]{D\lo{#1}{#2}(#3)\sb{#4}}
\newcommand{\DDY}[3]{\DD{#1}{\fY}{#2}{#3}}
\newcommand{\EE}[4]{E\lo{#1}{#2}(#3)\sb{#4}}
\newcommand{\EEY}[3]{\EE{#1}{\fY}{#2}{#3}}
\newcommand{\cF}[1]{\mathcal{F}\lo{#1}}
\newcommand{\hF}[1]{\Fil\sb{#1}}
\newcommand{\cK}[1]{\eK\lo{#1}}
\newcommand{\oK}{\bar{K}}

\newcommand{\hK}{\widehat{K}}
\newcommand{\hpsi}{\widehat{\psi}}
\newcommand{\hPsi}{\widehat{\Psi}}
\newcommand{\rhu}[1]{\rho\sp{#1}}

\newcommand{\sigu}[1]{\sigma\sp{#1}}

\begin{document}

\title{Homotopy types of diagrams of Chain Complexes}
\author{David Blanc}
\address{Department of Mathematics, University of Haifa, POB 3338, 3103301 Haifa, Israel}
\email{blanc@math.haifa.ac.il}
\author{Surojit Ghosh}
\address{Department of Mathematics, Indian Institute of Technology, Roorkee, Uttarakhand-247667, India}
\email{surojit.ghosh@ma.iitr.ac.in; surojitghosh89@gmail.com}
\author{Aziz Kharoof}
\address{Department of Mathematics, University of Haifa, POB 3338, 3103301 Haifa, Israel}
\email{azez.kh@gmail.com}
\date{\today}
\subjclass{Primary 18N60, secondary 18G35, 18G40, 55P99, 55S20, 55T99}
\keywords{chain complexes, directed diagrams, linear $\infty$-categories, higher homotopy invariants, spectral sequences}

\begin{abstract}
We study the homotopy theory of diagrams \w{\fX:\sI\to\Ch} of chain complexes over a field indexed by a finite poset $\sI$, and show that it can be completely described in terms of appropriate diagrams of graded vector spaces.
\end{abstract}
\maketitle

\setcounter{section}{-1}
\sect{Introduction}

Higher homotopy theory studies the higher homotopy invariant structure of $\infty$-cate\-go\-ries, model categories,
and  other similar frameworks. Such invariants include differentials in spectral sequences, higher order homotopy and cohomology operations, \w[-,]{A\sb{\infty}}  \w[-,]{L\sb{\infty}} and \ww{E\sb{n}}-structures on objects, and homotopy types of diagrams, all subsumed in principle under the  $k$-invariants of an \wwb{\infty,1}category (see \cite{DKSmitO,HNPrasA,BMeadA}).
However, in general it is impossible to reduce this higher structure to a $1$-categorical description.

One approach to exhibiting a collection of such higher invariants explicitly
is to use the homotopy categories of the functor categories \w{\cC\sp{\sI}} for all
\w{\sI\in\Cat} to encode the invariants of a model category $\cC$, say.  This idea was first
realized in Heller's notion of a \emph{hyperfunctor}:
a sort of functor from the $2$-category \w{\Cat} of all small categories to (large)
categories (see \cite{HellH}). Grothendieck later introduced the term \emph{derivators}
for the version encoding also (co)limits. Recently, Renaudin showed (in \cite{RenaP})
how reasonable model categories can be embedded in the category of derivators, and this was
extended to quasi-categories in \cite{ArliH} and \cite{FKRoveM} (with respect to
the simplicial enrichment of \cite{MRaptKD}).

In this paper, we provide an explicit approach to analyzing the homotopy types of
diagrams, in what appears to be the simplest non-trivial example:
the category \w{\Ch} of non-negatively graded chain complexes over
 a field. We consider only the ``linear $\infty$-category'' version
(enriched in unbounded chain complexes), and we restrict ourselves to indexing
categories $\sI$ which are finite partially ordered sets.  We are able to produce
an explicit (if complicated) description of the various homotopy categories of
\w{\Ch\sp{\sI}} only by making these very special assumptions. Nevertheless, we
believe that these may point the way to studying more general situations.

\begin{example}\label{egdifftwo}
The first example of a higher order invariant for chain complexes is as follows:
denote by \w{K(V,n)} the chain complex with homology $V$ concentrated in degree 
\w[,]{n\geq 0} and consider the commuting square in the homotopy category of chain complexes depicted in \wref[(a).]{eqsimplesec}  A cofibrant version of this diagram is provided by the solid outer square in \wref[(b),]{eqsimplesec} where \w{\binom{V'}{V}} denoting a chain complex with $V$ in degree $0$, \w{V'} in degree $1$, and the differential $d$  is an isomorphism (except for \w[,]{\binom{V_1\oplus V_2}{V}} where it is the fold map).  
The dashed lines represent the structure maps for the (homotopy) colimit of the top left corner, with $\psi$ the induced map into the target 
\w[.]{K(V,1)} Note that its domain is also weakly equivalent to \w[,]{K(V,1)} and rank of \w{H\sb{1}(\psi)} determines the weak homotopy type of the diagram.

\myqdiag[\label{eqsimplesec}]{
K(V,0) \ar[dd] \ar[rr] && 0\ar[dd] &&&& 
\binom{0}{V} \ar[dd] \ar[rrr] & & & \binom{V\sb{1}}{V} \ar@{-->}[dd]
\ar@/^2ex/[dddr]\\
&&&&&&&\\
0 \ar[rr] && K(V,1) &&&& \binom{V_2}{V} \ar@{-->}[rrr] \ar@/_2ex/[rrrrd] &&& \binom{V_1\oplus V_2}{V} \ar@{-->}[dr]^{\psi} & \\
&&&&&& &&&&  \binom{V_1\oplus V_2}{V} \simeq\binom{V}{0}\\
&(a)&&&&&&&(b)
}
\end{example}

We call the map $\psi$ in this example the \emph{value} of the \ww{d\sp{2}}-differential for 
the solid square \wref[(b).]{eqsimplesec} This is an example of the \emph{secondary invariants} for  a diagram of chain complexes, and as we shall see in Section \ref{cssd} below, the secondary data for the type of diagrams we study here can be described completely in terms of such differentials (and primary information in graded vector spaces).

What is more surprising, perhaps, is that \emph{all} the higher order structure
can be reduced to secondary structure (in more complicated diagrams),
so that it is ultimately describable in terms of such simple \ww{d\sp{2}}-differentials.
The reader is invited to look at the detailed example of how this is done for
what one would expect to be a \ww{d\sp{3}}-differential (involving a $3$-dimensional cube),
in \S \ref{sthreecube} below (see in particular \S \ref{sthreediff}).

\begin{mysubsection}{Overall strategy}
\label{soas}
We determine the homotopy type of the given diagram \w{\fX:\sI\to\Ch} by means of two sequences of diagrams: the \emph{derived diagrams} \w{\Der{k}\fX} \wb{k\geq 0} in 
graded vector spaces, which we think of as being purely formal, and the \emph{hybridizations} \w{\Hyb{k}\fX} \wb{k\geq 0} in the category \w{\Chp} (see \S \ref{snac} below): these are diagrams having maps of graded vector spaces (possibly of positive degree) up to dimension $k$, and maps of chain complexes in higher dimensions.  These two types of diagrams are constructed inductively, beginning with 
\w{\Der{0}\fX=H\sb{0}\fX} and \w[.]{\Hyb{0}\fX=\fX} The hybridization \w{\Hyb{k}\fX} is obtained from \w{\Hyb{k-1}\fX} through several intermediate stages:

\begin{enumerate}
\renewcommand{\labelenumi}{(\alph{enumi})~}
\item Assume by induction that we have obtained from \w{\fX} a \wwb{k-1}hybrid diagram 
\w[,]{\fY=\Hyb{k-1}\fX:\sIu{k-1}\to\Chp} with the property that $\fX$ itself can be recovered up to weak equivalence from $\fY$.
\item From $\fY$ we construct its derived diagram \w[,]{\Der{k}\fX:\hIDu{k}\to\grV}
 taking values in graded vector spaces (allowing maps of positive
degree). The indexing categories \w{\hIDu{k}} are themselves partially ordered sets describable in terms of the original indexing category $\sI$.
\item  From \w{\Der{k}\fX} we obtained the \emph{expanded} derived diagram 
\w{\wY:\sJ\sb{\Ind}\bup{2}\to\Chp} from the \ww{\Ind}-completion of \w{\sJ:=\sIu{k-1}} (which is no longer formal, and whose \wwb{k+1}truncation is weakly equivalent to that of the original $\fY$).
\item Finally, from $\wY$ and $\fY$ we produce the next stage in the induction, the \emph{reconstructed} $k$-hybrid diagram \w[.]{\Hyb{k}\fX:\sIu{k}\to\Chp}
\end{enumerate}
\end{mysubsection}

In terms of these constructions \wh the details of which are given in Section \ref{cdd} \wh our main result states:

\begin{thma}
Given a diagram \w{\fX:\sI\to\Ch} in non-negatively graded chain complexes indexed by a finite partially ordered set, for each \w{k\geq 0} the associated derived diagram \w[,]{\Der{k}\fX:\hIDu{k}\to\grV} in graded vector spaces, determines the homotopy type of the $k$-truncation of $\fX$.
\end{thma}
\noindent See Theorem \ref{thybrdztn} below\vsm.

Thus the sequence \w{(\Der{k}\fX)\sb{k=0}\sp{\infty}} encodes the weak homotopy type
of the original diagram $\fX$. 

\begin{notn}\label{snac}
We denote by \w{\Vect} the category of vector spaces over a fixed ground field $\bF$, by
\w{\grV} that of non-negatively graded objects over \w{\Vect} (with linear maps
of degree \w[),]{\geq 0} by \w{\Ch} that of non-negatively graded chain complexes
over \w[,]{\Vect} and by \w{\Chp} the extension of \w{\Ch} allowing also maps of
positive degree (possibly non-homogeneous). Note that \w{\grV} embeds in \w{\Chp}
as a full subcategory, and we occasionally further embed \w{\Vect} in \w{\grV} in dimension $0$. 
%It is sometimes convenient to think of a diagram
%\w{\fX:\sI\to\Chp} or to \w{\grV} as being indexed by a (non-negatively) graded category \w[,]{\sI\sb{\ast}} which %allows us to keep track of the degrees of the maps. 

The \emph{cone} \w{CX} on a chain complex $X$ and  the \emph{suspension} \w{\Sigma X} are defined as usual (see
\cite[\S 1.5]{WeibHA}), and thus by induction we define \w{C^k X:= C(C^{k-1}X)} and 
\w[.]{\Sigma^k X:= \Sigma (\Sigma^{k-1}X)} The \emph{reduced suspension} 
\w{\tS X} is given by shifting $X$ upwards by one degree.

A finite partially ordered set $\sI$ is called a \emph{lattice}
if it has sets \w{\cF{0}:=\{\alpha\sb{\init}\sp{i}\}\sb{i=1}\sp{n\sb{\init}}} of weak
\emph{initial objects} (that is, minima) and  \w{\{\omega\sb{\term}\sp{j}\}\sb{j=1}\sp{n\sb{\term}}} of
a weak \emph{terminal objects} (maxima). Such a lattice has a finite filtration
\w{\cF{0}\subseteq \cF{1}\dotsc\subseteq \cF{N}=\sI}
given by distance from \w[.]{\cF{0}}

For any \w{\alpha,\beta\in\Obj(\sI)} in such a lattice, we denote by \w{\cib{\beta}} the sub-poset of all (weak) predecessors of $\beta$ (including $\beta$ itself), by \w{\cibb{\beta}} the sub-poset of strict predecessors, and by \w{\aci{\alpha}} the sub-poset of all (weak) successors of $\alpha$ (including $\alpha$ itself). 
If \w[,]{\alpha<\beta} the categories \w{\aib{\alpha}{\beta}} and \w{\aibb{\alpha}{\beta}} are defined analogously.
Similarly for \w{\aib{\alpha}{\Gamma}} and \w{\aibb{\alpha}{\Gamma}} when 
\w{\Gamma=\{\gamma\sb{1},\dotsc,\gamma\sb{m}\}} is a collection of objects such that \w{\alpha<\gamma\sb{i}} for each \w[.]{1\leq i\leq m}
\end{notn}

\begin{mysubsection}{Organization}
\label{sorg}
Section \ref{cscc} provides background and establishes terminology on the category of diagrams of chain complexes. Section \ref{cssd} analyzes the secondary structure of such diagrams. Section \ref{cdd} defines the derived diagram and studies its basic properties, and proves Theorem A. 
Section \ref{cetoi} provides a template for the full description of two infinite sequences of 
diagrams, while Section \ref{css} explains how to apply our methods to the spectral sequence of a filtered chain complex.  Finally, Section \ref{cfd} indicates how one might deal with more general diagrams.
\end{mysubsection}

%
%c1   Diagrams of chain complexes
%
\sect{Diagrams of chain complexes}
\label{cscc}

Our main object of study in this paper are diagrams of chain complexes \w{\fX:\sI\to\Ch} indexed by
a lattice $\sI$ (see \S \ref{snac}). Although our results can ultimately be placed in the context of ``linear $\infty$-categories'' \wh that is, categories enriched in (unbounded) chain complexes \wh the work itself is most easily carried out in the framework of the projective model structure on  \w{\Ch\sp{\sI}} (see \cite[\S 11.6]{HirM}), induced from the usual model structure on \w{\Ch} (see \cite[\S 7]{DSpaH}). First,  we recall the few necessary ingredients we need on the  model structures on diagrams in chain complexes:

\begin{defn}\label{dmincof}
A diagram \w{\fX:\sI\to\Ch} as above is called \emph{minimally cofibrant} if
\begin{enumerate}
\renewcommand{\labelenumi}{(\alph{enumi})~}
\item For each \w[,]{\alpha\in\Obj(\sI)} \w{\fX(\alpha)} can be split (non-canonically) as a direct sum of
complexes of the form \w{K(V,m)=\tS\sp{m}V} (with $V$ in dimension $m$ and $0$
elsewhere) or \w{CK(V,m)} (which has copies of $V$ in dimensions $m$ and \w[,]{m+1}
and $0$ elsewhere, with \w[),]{\pd{m+1}=\Id\sb{V}} for some
\w{V\in\Vect} and \w[.]{m\geq 0}
\item For every summand of \w{\fX(\beta)} of the form \w[,]{CK(V,m)} there is a
  (non-identity) map \w{\phi:\alpha\to\beta} in $\sI$ such that \w{\fX(\phi)} induces
  the inclusion \w{K(V,m)\hra CK(V,m)} for some summand \w{K(V,m)} in
  \w{\fX(\alpha)} (so in particular this can only happen if $\beta$ is not initial).
\item For each non-initial \w[,]{\beta\in\Obj(\sI)} the structure map
\w{\colim\sb{\cibb{\beta}}\,\fX\to\fX(\beta)} is monic. 
\end{enumerate}
\end{defn} 

\begin{mysubsection}{The $\Ind$-completion}
\label{sindcomp}
For any cofibrant diagram \w[,]{\fX:\sI\to\Ch} we expect the extra information
(beyond its \emph{homotopy diagram} \w[)]{\Hs\fX} to be encoded by its
\ww{\Ind}-\emph{completion} \wh that is, the diagram obtained by repeatedly adding
formal colimits $\check{\beta}$ for all sub-indexing categories 
\w{\aibb{\alpha}{\beta}} in $\sI$ to obtain \w[,]{\hII} and then extending $\fX$ to a colimit-preserving functor \w[.]{\wXi:\hII\to\Ch} 
The weak homotopy types of $\fX$ and \w{\wXi} determine each other uniquely.
Note, however, that it is quite difficult to extract this extra information from \w{\wXi} directly; our main goal in this paper is to provide a more explicit inductive description
of this data.

The usual \ww{\Ind}-completion is defined iteratively, by adding formal colimits, then colimits of the new diagrams formed, and so on.  However, for our purposes we shall be mostly interested only in the first stage, where we only add colimits of finite directed  sub-diagrams of $\sI$ (multiple pushouts), and stop there, yielding 
\w[.]{\wXn{2}:\hIIn{2}\to\Ch} 
\end{mysubsection}

\begin{defn}\label{dtrunc}
The $k$-\emph{truncation} functor \w[,]{\tau\sb{k}\fX:\sI\to\Chp} applied to a diagram \w[,]{\fX:\sI\to\Chp} is given by
\begin{myeq}\label{eqtrunc}
[\tau\sb{k}\fX(\alpha)]\sb{i}~:=~\begin{cases}
  \fX(\alpha)\sb{i} & \text{if}\ i<k\\
  \Cok(d\sb{k+1}:\fX(\alpha)\sb{k+1}\to\fX(\alpha)\sb{k}) & \text{if}\ i=k\\
  0 & \text{if}\ i>k,
  \end{cases}
\end{myeq} 
\noindent with the natural fibration \w[.]{p\sb{k}:\fX\to\tau\sb{k}\fX}

The \emph{\wwb{k-1}connected cover} \w{\fX\lra{k}} of $\fX$ has
\begin{myeq}\label{eqconnc}
  [\fX\lra{k}(\alpha)]\sb{i}~:=~\begin{cases}
  \fX(\alpha)\sb{i} & \text{if}\ i>k\\
  Z\sb{k}\fX(\alpha):=\Ker(d\sb{k}:\fX(\alpha)\sb{k}\to\fX(\alpha)\sb{k-1}) &
  \text{if}\ i=k\\
  0 & \text{if}\ i<k,
  \end{cases}
\end{myeq} 
\noindent with the natural cofibration \w[.]{\iota\sb{k}:\fX\lra{k}\to\fX}

We have a natural homotopy (co)fibration sequence
\mydiagram[\label{eqcofibs}]{
  H\sb{k+1}\fX \cong \tau\sb{k+1}\fX\lra{k}
  \ar[rr]\sp-{\tau\sb{k+1}\iota\sb{k}} &&
  \tau\sb{k+1}\fX \ar[r]\sp{p'\sb{k+1}} & \tau\sb{k}\fX \ar[r]\sp(0.4){q\sb{k}} &
  \Sigma H\sb{k+1}\fX
}
\noindent in \w[.]{\Ch\sp{\sI}}
\end{defn}

\begin{defn}\label{dformal}
A diagram \w{\fX:\sI\to\Chp} is $k$-\emph{formal} if
\w{\fX(\alpha)\sb{i}\cong H\sb{i}\fX(\alpha)} for all \w{i<k} and \w[,]{\alpha\in\sI} so \w{\iota\sb{k}:\fX\lra{k}\to\fX} is an isomorphism in dimension $k$,  and thus
\w{\iota\sb{k}} has a canonical retraction \w[.]{\rho\sb{k}:\fX\to\fX\lra{k}}
We allow the maps of \w{\tau\sb{k}\fX} to be of positive degree in \w{\grV}
(as long as they land in dimensions \w[),]{<k} but require \w{\fX\lra{k}} to land
in the usual \w[.]{\Ch} 

The diagram \w{\fX:\sI\to\Chp} is called $k$-\emph{hybrid} if it is $k$-formal,
and \w{\fX\lra{k}:\sI\to\Ch} is minimally cofibrant. 
\end{defn}

\begin{mysubsection}{Primary structure of a diagram}
\label{spsd}
The underlying \emph{primary structure} of a diagram \w{\fX:\sI\to\Ch} as above is given by  
\w[.]{\Hs\fX:\sI\to\grV} Note that this suffices to determine the \emph{primary homotopy type}
\w{\pi\sb{0}\fX} of $\fX$ in\w[,]{(\ho\Ch)\sp{\sI}} namely, the \emph{isomorphism type} of 
\w[,]{\Hs\fX} but describing the latter succinctly involves delicate questions of quiver representation
theory, which we shall not attempt to address here (see \cite{GabrU}).
\end{mysubsection}

\begin{remark}\label{roas}
As noted in \S \ref{soas}, our determination of the homotopy type of the given diagram \w{\fX:\sI\to\Ch} is by means of two sequences of diagrams: the derived diagrams \w{\Der{k}\fX} \wb{k\geq 0} in 
graded vector spaces, and the hybridizations \w{\Hyb{k}\fX} \wb{k\geq 0} in chain complexes.

Ideally, we would have liked our approximation at level $n$ to contain all the $n$-th
order information about $\fX$, starting with \w{\Hs\fX} globally encoding all the primary data.  However, this requires taking more seriously the view of \w{\Ch} as
a ``linear \wwb{\infty,1}category'' \wh that is, a category enriched in (unbounded) chain complexes, including a version of $k$-truncation of the enrichment which captures
the idea of the \emph{$k$-stem} of an unbounded complex (see \cite{BBlaS} and compare
\cite{BBChorT}). In particular, we need to keep track of the various ways higher 
truncations may map to lower ones, as well as the fact that 
\w[.]{\Hom(\Sigma A,B)\cong\Omega\Hom(A,B)} Thus we prefer the more pedestrian but simpler approach of using an induction on the truncations of the chain complexes themselves (which eventually captures all of the same information, but at a slower pace).
\end{remark}

%
%c2   Secondary structure of diagrams
%
\sect{Secondary structure of diagrams}
\label{cssd}
Our analysis of the higher order structure in diagrams of chain complexes begins with the secondary 
structure.  As we shall see, not only can this be described explicitly, it is in fact all that is needed to capture all the higher structure, by means of our inductive process.

The basic example of secondary structure arises from the fact that the homotopy pushout of two maps from \w{K(V,k)} to zero objects is \w[,]{K(V,k+1)} which yields the example described in \S \ref{egdifftwo}. For the general case, we need to understand how such pairs of zero maps can be related to each other.

\begin{mysubsection}{Pairs of objects}
\label{spairs}
The general form of sub-diagram needed to describe the full secondary structure will consist of a collection of ``zero squares'' as 
\wref[(a),]{eqsimplesec} all with the same source (say, \w[).]{K(V,k)} 

To analyze this situation, assume given a partially ordered set $\sI$ and a diagram 
\w{\fX:\sI\to\Ch} as above. We fix a specific dimension \w{k\geq 0} to work in, and for any \w{\alpha<\beta} in $\sI$, let
$$
K\sp{\alpha}\sb{\beta}=K\sp{\alpha}\sb{\beta}(\fX):=\Ker(H\sb{k}\fX(\alpha)\to H\sb{k}\fX(\beta))~.
$$

We denote by \w{\Incomp(\sI)} the set of \emph{incomparable} objects  in $\sI$ \wh that is, pairs \w{\{\gamma,\delta\}} of objects for which \w{\gamma\not\leq\delta} and \w[.]{\delta\not\leq\gamma} For any \w{\{\gamma,\delta\}} in 
\w[,]{\Incomp(\aci{\alpha})} let 
$$
K\sp{\alpha}\sb{\{\gamma,\delta\}}=K\sp{\alpha}\sb{\{\gamma,\delta\}}(\fX):=K\sp{\alpha}\sb{\gamma}(\fX)\cap K\sp{\alpha}\sb{\delta}(\fX)
$$
(which we shall usually abbreviate to 
\w[,]{K\sb{\{\gamma,\delta\}}} and so on).

Let \w{\aib{\alpha}{\{\gamma,\delta\}}} be the subcategory of $\sI$ consisting of
$\alpha$, $\gamma$, and $\delta$, and the unique paths between them. The (homotopy) colimit of the restriction of $\fX$ to 
\w{\aib{\alpha}{\{\gamma,\delta\}}} is a (homotopy) pushout.
If we further restrict $\fX$ (in degree $k$) to
\w[,]{K\sb{\{\gamma,\delta\}}} this homotopy pushout equals
\w[.]{\Sigma K\sb{\{\gamma,\delta\}}}
 
Thus if \w{\beta>\gamma,\delta} is any join of an incomparable pair 
\w[,]{\{\gamma,\delta\}} we have an induced map 
 \w{\psi=\psi\sb{(\gamma,\delta;\beta)}:\Sigma K\sb{\{\gamma,\delta\}}\to
   H\sb{k+1}\fX(\beta)} as in \wref[(b),]{eqsimplesec}
 which records the fact that we have (possibly non-trivial) secondary data for $\fX$ restricted to \w{\aib{\alpha}{\beta}} in degree \w[.]{k+1}
\end{mysubsection}

\begin{mysubsection}{Three zero objects}
\label{sthreeo}
In order to understand the complications that may arise in the general case, we first describe the example of three zero maps. Thus, consider the following poset $\sI$:
\myudiag[\label{eqcubeI}]{
 & \gamma \ar[rd] & \\
\alpha \ar[r]  \ar[dr] \ar[ur] & \delta\ar[r] &\beta~. \\
 & \epsilon \ar[ur] & 
}
\noindent Let \w{K\sb{\{\gamma,\delta,\epsilon\}}\subseteq H\sb{k}(\fX(\alpha))}
be the common kernel of the induced maps from \w{H\sb{k}(\fX(\alpha))} to \w[,]{H\sb{k}(\fX(\gamma))} \w[,]{H\sb{k}(\fX(\delta))}  and 
\w[.]{H\sb{k}(\fX(\epsilon))}
We can split it off the three kernels of the pairs of \S \ref{spairs} \wh e.g.
\w[.]{K\sb{\{\gamma,\delta,\epsilon\}}\oplus\widetilde{K}\sb{\{\gamma,\delta\}}
=K\sb{\{\gamma,\delta\}}}
This yields a description of the (homotopy) colimit of \w[,]{\fX'} 
the relevant subcomplexes of the restriction of $\fX$ to 
\w[:]{\aibb{\alpha}{\beta}}
\mytdiag[\label{eqcubeH}]{
\fX'(\alpha)=V:=K\sb{\{\gamma,\delta,\epsilon\}} \oplus\widetilde{K}\sb{\{\gamma,\delta\}}
\oplus\widetilde{K}\sb{\{\gamma,\epsilon\}}\oplus\widetilde{K}\sb{\{\delta,\epsilon\}}
\ar[rr]  \ar[dr] \ar[dd] && 
C(V\setminus\widetilde{K}\sb{\{\delta,\epsilon\}}) \ar[dr] \ar[dd] \\ 
& C(V\setminus\widehat{K}\sb{\{\gamma,\epsilon\}})\ar[dd] \ar[rr] && 
\Sigma K\sb{\{\gamma,\delta\}} \ar[dd] \\ 
C(V\setminus\widetilde{K}\sb{\{\gamma,\delta\}}) \ar[dr] \ar[rr] && 
\Sigma K\sb{\{\gamma,\epsilon\}}\ar[dr] \\ & 
\Sigma K\sb{\{\delta,\epsilon\}} \ar[rr] && \hocolim\sb{\aibb{\alpha}{\beta}} \fX'.
}

Note that in order to make this diagram cofibrant, we must add the relevant
complements of the kernels \wh for example, in the top square we should have
$$
\xymatrix@R=10pt@C=15pt{
V\ar[rr] \ar[d] && 
\widehat{K}\sb{\{\delta,\epsilon\}}\oplus C(V\setminus\widetilde{K}\sb{\{\delta,\epsilon\}})
\ar[d] \\ 
\widehat{K}\sb{\{\gamma,\epsilon\}}\oplus
C(V\setminus\widetilde{K}\sb{\{\gamma,\epsilon\}})\ \ar[rr] && 
C\widetilde{K}\sb{\{\gamma,\epsilon\}}\oplus C\widetilde{K}\sb{\{\delta,\epsilon\}}
\oplus\Sigma K\sb{\{\gamma,\delta\}}\simeq\Sigma K\sb{\{\gamma,\delta,\epsilon\}} \oplus
\Sigma \widetilde{K}\sb{\{\gamma,\delta\}}~. 
}
$$
\noindent We therefore have a non-canonical equivalence
\begin{myeq}\label{eqnoncan}
  \hocolim\sb{\aibb{\alpha}{\beta}}\fX'~\simeq~\Sigma K\sb{\{\gamma,\delta,\epsilon\}}
  \oplus\Sigma K\sb{\{\gamma,\delta,\epsilon\}}\oplus \Sigma\widetilde{K}\sb{\{\gamma,\delta\}}
 \oplus \Sigma\widetilde{K}\sb{\{\gamma,\epsilon\}}
 \oplus \Sigma\widetilde{K}\sb{\{\delta,\epsilon\}}~.
\end{myeq}
\noindent Thus the map 
 \w{\psi\sb{(\gamma,\delta,\epsilon;\beta)}} from 
 \w{H_{k+1}(\hocolim\sb{\aibb{\alpha}{\beta}}\fX')} 
 to \w{H\sb{k+1}\fX(\beta)} is induced by the \emph{compatible} maps
 \w[,]{\psi\sb{(\gamma,\delta;\beta)}}  \w[,]{\psi\sb{(\gamma,\epsilon;\beta)}}
 and  \w[.]{\psi\sb{(\delta,\epsilon;\beta)}}

 Note, however, that the splittings above involve choices; in our general description we shall avoid 
 making these choices at the cost of using a more complicated indexing category (see \S \ref{sgdd} below). 
\end{mysubsection}

\begin{mysubsection}{Describing the homotopy colimit}
\label{sdhc}
As noted above, the identification of the homotopy colimit in \wref{eqnoncan} is
non-canonical. In order to obtain a canonical description, note that if $\fX$ is (minimally) cofibrant (see
\S \ref{dmincof}), the homotopy colimit \w{A= \hocolim\sb{\aibb{\alpha}{\beta}}\fX}
is just the colimit. Assuming for simplicity that 
$$
V~:=~K\sb{\{\gamma,\delta,\epsilon\}}~=~K\sb{\{\gamma,\delta\}}~=~
K\sb{\{\gamma,\epsilon \}}~=~K\sb{\{\delta,\epsilon\}}~,
$$
as a chain complex, this colimit is given by
\begin{myeq}\label{eqhocopair}
A~=~(V\sb{\gamma}\oplus V\sb{\delta}\oplus V\sb{\epsilon)}~\xra{\nabla}~V~,
\end{myeq}
\noindent where all four vector spaces are copies of $V$ and $\nabla$ is the fold map (so
\w[,]{A\simeq\Sigma\sp{k+1}V\vee\Sigma\sp{k+1}V} as in \wref[).]{eqnoncan}

This fits into a diagram of chain complexes:
\myrdiag[\label{eqtrincexc}]{
&& \Sigma K\sb{\{\gamma,\delta\}} \ar[rrd] && \\
\Sigma K\sb{\{\gamma,\delta,\epsilon\}}\ar[rru]^{\iota\sb{(\gamma,\delta)}}
\ar[rr]^(0.7){\iota\sb{(\gamma,\epsilon)}} \ar[rrd]_{\iota\sb{(\delta,\epsilon)}} &&
\Sigma K\sb{\{\gamma,\epsilon\}} \ar[rr] && \hocolim\sb{\aibb{\alpha}{\beta}}\fX\\
&& \Sigma K\sb{\{\delta,\epsilon\}} \ar[rru] &&
}
\noindent which we shall indicate more succinctly by the linear diagram:
\begin{myeq}\label{eqincexc}
  \Sigma K\sb{\{\gamma,\delta,\epsilon\}}~\xra{\theta}~\Sigma K\sb{\{\gamma,\delta\}}
  \oplus\Sigma K\sb{\{\gamma,\epsilon\}}\oplus \Sigma K\sb{\{\delta,\epsilon\}}~\xra{\psi}~
  \hocolim\sb{\aibb{\alpha}{\beta}}\fX~
\end{myeq}
\noindent(using the fact that the biproduct $\oplus$ in \w{\Ch} is both a product and coproduct).

Note that \wref{eqincexc} becomes a short exact sequence upon applying \w{H\sb{k+1}} (and has trivial homology in other dimensions). 
Here \w{\Sigma K\sb{\{i,j\}}} is the chain complex
\begin{myeq}\label{eqpairs}
(V\sb{i}\oplus V\sb{j})~\xra{\nabla}~V
\end{myeq}
\noindent for each pair \w[,]{\{i,j\}\subseteq\{\gamma,\delta,\epsilon\}}
with each \w[,]{\Sigma K\sb{\{i,j\}}\simeq\Sigma\sp{k+1}V} 
and \w{\Sigma K\sb{\{\gamma,\delta,\epsilon\}}} is the chain complex kernel of $\psi$, viz.,
$$
(V'\sb{\gamma}\oplus V'\sb{\delta}\oplus V'\sb{\epsilon})~\xra{\partial}~
(V'\sb{\gamma,\delta}\oplus V'\sb{\delta,\epsilon})
$$
\noindent (all copies of $V$), with $\partial$ mapping the first summand to the first,
the last to the last, and the middle summand by the diagonal. Thus 
\w[.]{\Sigma K\sb{\{\gamma,\delta,\epsilon\}}\simeq\Sigma\sp{k+1}V} The map $\psi$
takes \w{V'\sb{i}} to the two copies of \w{V\sb{i}} in \wref{eqpairs} by
the signed diagonal map \w[.]{(1,-1)} 
\end{mysubsection}

\begin{mysubsection}{Multiple pairs}
\label{smulpr}
More generally, if we have $m$ incomparable objects \w{\{\gamma\sb{i}\}\sb{i=1}\sp{m}} in $\sI$ (see \S \ref{spairs})
with a common join $\beta$, let $K$ denote the common kernel of the maps
\w[.]{H\sb{k}\fX(\alpha)\to H\sb{k}\fX(\gamma\sb{i})} 
Assuming $\fX$ is as in \S \ref{dmincof}, \w{\fX(\gamma\sb{i})} is the cone
\w[,]{(K\sb{i}\to K)} so that
\w{\hocolim\sb{\aibb{\alpha}{\beta}}\fX} is the chain complex
 \w[,]{\bigoplus_{i=1}^{n}\ K\sb{i} \to K} non-canonically weakly equivalent to
 \w[.]{\bigvee_{i=1}^{n-1}\ \Sigma K}

The canonical description is again given by an inclusion-exclusion sequence of
chain complexes as in \wref[:]{eqincexc}
\begin{myeq}\label{eqmultiple}
\begin{split}
\Sigma K\sb{\{\gamma\sb{1},\dotsc,\gamma\sb{m}\}}~&
\to~\bigoplus\sb{1\leq i\leq m}~
 \Sigma K\sb{\{\gamma\sb{1},\dotsc,\widehat{\gamma\sb{i}},\dotsc,\gamma\sb{m}\}}~
 \to~\dotsc\to\\
 \bigoplus\sb{1\leq i<j<k\leq m}\, &
 \Sigma K\sb{\{\gamma\sb{i},\gamma\sb{j},\gamma\sb{k}\}} \ \to \  
 \bigoplus\sb{1\leq i<j\leq m}\ \Sigma K\sb{\{\gamma\sb{i},\gamma\sb{j}\}}
%\hocolim\sb{\aib{\alpha}{\{\gamma\sb{i},\gamma\sb{j}\}}}\fX~
\xra{\psi}~
 \hocolim\sb{\aibb{\alpha}{\beta}}\fX~.
\end{split}
\end{myeq}
\noindent which becomes an exact sequence of vector spaces after applying
\w{H\sb{k+1}} (this follows by induction from the fact that the alternating sum
of the binomial coefficients is zero).
\end{mysubsection}

%
%c3     Derived diagrams
%
\sect{Derived diagrams}
\label{cdd}

We now describe our main tool in analyzing the higher homotopy information of a diagram \w[,]{\fX:\sI\to\Ch} where $\sI$ is a partially ordered set.  
As explained in \S \ref{soas}, our approach is inductive, producing two sequences of
approximations. Thus for each \w{k\geq 0} we have:

\begin{enumerate}
\renewcommand{\labelenumi}{(\alph{enumi})~}
\item The $k$-\emph{hybridization} \w{\Hyb{k}\fX:\sIu{k}\to\Chp} which retains all 
information about the weak homotopy type of $\fX$, in a form which is partly formal
(that is, a diagram of graded vector spaces), at the expense of making the indexing category \w{\sIu{k}} more complicated. 
\item The $k$-th \emph{derived diagram} \w[,]{\Der{k}\fX:\hIDu{k}\to\grV} which is
purely formal, and encodes all homotopy information about $\fX$ through degree $k$.
\end{enumerate}

Here the sequence \w{(\Der{k}\fX)\sb{k=0}\sp{\infty}} is our ultimate product, expressing the weak homotopy type of $\fX$ in terms of graded vector spaces. The sequence \w{(\Hyb{k}\fX)\sb{k=0}\sp{\infty}} is an auxiliary
construction, which allows us to pass from \w{\Der{k}\fX} to \w[.]{\Der{k+1}\fX} 

This section is devoted to a single step in this inductive process, starting from 
\w{\Hyb{k}\fX} and passing to \w[.]{\Hyb{k+1}\fX} Thus we assume given a $k$-formal diagram \w{\fY:\sJ\to\Chp} (of the form \w[,]{\fY=\Hyb{k}\fX} though this is not
used anywhere). Because the model category of (unbounded) chain complexes  is stable,
the value of $k$ does not matter, so for simplicity we may take \w[.]{k=0} The indexing category $\sJ$, which we think of as \w[,]{\sIu{k}} is also partially ordered (although it is more complicated than the original $\sI$).

\begin{defn}\label{dmaxcoll}
Assume given \w[,]{k\geq 0} two objects \w{\alpha<\beta} in a finite partially
ordered set $\sJ$, and a $k$-formal diagram \w[.]{\fY:\sJ\to\Ch} For each collection
\w{\Gamma=\{\gamma\sb{i}\}\sb{i=1}\sp{m}} of incomparable objects in 
\w{\ajb{\alpha}{\beta}} 
\wb[,]{m\geq 2}let \w{K\sp{\alpha}\sb{\gamma\sb{i}}:=\Ker(H\sb{k}\fX(\alpha)\to H\sb{k}\fX(\gamma\sb{i}))} and 
more generally \w{K\sp{\alpha}\sb{\Gamma'}:=\bigcap\sb{\gamma\in\Gamma'}\,K\sp{\alpha}\sb{\gamma}} for any subset \w[.]{\Gamma'\subseteq\Gamma}

The \emph{partial derived diagram \w{\fY'(\alpha,\Gamma,\beta)} of  $\fY$ at \w{(\alpha,\Gamma,\beta)}} is defined to be the following diagram of graded vector spaces: 
\begin{myeq}\label{eqdergamma}
K\sp{\alpha}\sb{\Gamma} \to~\bigoplus\sb{1\leq i\leq m}~
K\sp{\alpha}\sb{\Gamma\setminus\{\gamma\sb{i}\}}\to~\dotsc\to \bigoplus\sb{1\leq i<j\leq m}\
K\sp{\alpha}\sb{\{\gamma\sb{i},\gamma\sb{j}\}}~\xra{\Psi:=\bot\sb{i<j}\psi\sb{i,j}}~
H\sb{k+1}\fY(\beta)
\end{myeq}
\noindent (compare \wref[),]{eqmultiple} where the first term is in degree $k$, and the rest in degree \w[,]{k+1} so all but the first map are of degree $0$. 

As in \wref[,]{eqincexc} we use the fact that $\oplus$ is a biproduct in \w{\grV} to condense a commutative diagram in which each summand in \wref{eqdergamma} appears separately, as do the individual components
of the maps between summands: see \wref{eqdergammak} below.

The structure maps
\w[,]{\hdt{i,j}:K\sp{\alpha}\sb{\Gamma}\to K\sp{\alpha}\sb{\{\gamma\sb{i},\gamma\sb{j}\}}}
all monomorphisms of degree \w[,]{+1} are called 
\emph{formal differentials} of $\fY$.
Each map \w{\psi\sb{i,j}:K\sp{\alpha}\sb{\{\gamma\sb{i},\gamma\sb{j}\}}\to H\sb{k+1}\fY(\beta)} is called the \emph{evaluation map} at
\w[,]{(\gamma\sb{i},\gamma\sb{j})} with
\w{\dt{i,j}:=\psi\sb{i,j}\circ \hdt{i,j}} the \emph{value} of the differential at \w{(\gamma\sb{i},\gamma\sb{j})} (also of degree \w[).]{+1}

The fact that $\Psi$ fits into the commutative
diagram \wref{eqdergamma} implies that there is coordination between the different values
(since the homotopy colimit of the slice diagram over \w{H\sb{k+1}\fX(\beta)}
is non-canonically equivalent to \w{m-1} distinct copies of \w[,]{K\sp{\alpha}\sb{\Gamma}} by \S \ref{smulpr}).
\end{defn}

\begin{mysubsection}{Global derived diagrams}
\label{sgdd}
Given a partially ordered set $\sJ$, let \w{\Path\sb{\sJ}} denote the category whose objects are
factorizations of \w{f:\alpha\to\beta} via $m$ incomparable intermediate objects
\w{\Gamma=\{\gamma\sb{i}\}\sb{i=1}\sp{m}} \wb[.]{m\geq 2} Its objects are \w{(\alpha,\Gamma,\beta)} as in \S \ref{dmaxcoll}, and we have a unique morphism
\w{(\alpha,\Gamma',\beta')\to(\alpha,\Gamma,\beta)} in \w{\Path\sb{\sJ}} (for \w[)]{\Gamma'=\{\gamma\sb{j}'\}\sb{j=1}\sp{n}} whenever \w{\beta'\leq\beta} and \w{\Gamma'\leq\Gamma} (that is, for each \w{1\leq j\leq n} there is an \w{1\leq i\leq m} with \w{\gamma'\sb{j}\leq\gamma\sb{i}} in $\sJ$ \wh so in particular, we could have 
\w[).]{\Gamma'\subseteq\Gamma} 

An example of a morphism (from top to bottom):
\mysdiag[\label{eqmorpath}]{
&\gamma'\sb{j\sb{1}}\ar|(.52){\hole}[dd]\ar[rrrd]&&& \\ 
\alpha\ar[dd]_{=} \ar[ru] \ar[rr] &&\gamma'\sb{j\sb{2}}\ar[dd]\ar[rr]&& \beta'\ar[dd]\\
&\gamma\sb{i\sb{1}}\ar|(.37){\hole}[rrrd]&&& \\ 
\alpha\ar[rd]\ar[ru] \ar[rr] &&\gamma\sb{i\sb{2}}\ar[rr]&& \beta\\
&\gamma\sb{i\sb{3}}\ar[rrru]&&& 
}

We denote by \w{\wPath\sb{\sJ}} the analogous category of systems \w{(\alpha,\Gamma)} (omitting the common join $\beta$ for $\Gamma$, which is still required to exist). There is a forgetful functor 
\w[.]{U:\Path\sb{\sJ}\to\wPath\sb{\sJ}}

For every morphism \w{\iota\sp{\Gamma'}\sb{\Gamma}:(\alpha,\Gamma',\beta')\to(\alpha,\Gamma,\beta)} 
in \w{\Path\sb{\sJ}} we have a commuting diagram
\mydiagram[\label{eqsquare}]{
K\sp{\alpha}\sb{\Gamma'}\ar[rrd]_{\hdt{i\sb{1},i\sb{2}}}
\ar[rr]^{\hdt{j\sb{1},j\sb{2}}} && K\sp{\alpha}\sb{\{\gamma'\sb{j\sb{1}},\gamma'\sb{j\sb{2}}\}} \ar[d]\sb{\varphi} \ar[rr]\sp{\psi'\sb{j\sb{1},j\sb{2}}}&& 
H\sb{k+1}\fY(\beta')\ar[d]\sp{H\sb{k+1}\fY}\\
&& K\sp{\alpha}\sb{\{\gamma\sb{i\sb{1}},\gamma\sb{i\sb{2}}\}} \ar[rr]\sb{\psi\sb{i\sb{1},i\sb{2}}}
&& H\sb{k+1}\fY(\beta)
}
\noindent relating the values of the differentials. We also have a functor
\w{\eK:\wPath\sb{\sJ}\to\Vect} sending a morphism
\w{(\alpha,\Gamma')\to(\alpha,\Gamma)} in \w{\wPath\sb{\sJ}} to the left triangle in \wref[.]{eqsquare}
\end{mysubsection}

\begin{defn}\label{ddercat}
The \emph{derived category} \w{\hJD} of $\sJ$ is defined to be the union of the following subcategories, related by the 
specified morphisms (with the obvious compositions):

\begin{enumerate}
\renewcommand{\labelenumi}{(\alph{enumi})~}
\item The given category $\sJ$;
\item For each \w{(\alpha,\Gamma,\beta)} in \w[,]{\Path\sb{\sJ}} the indexing category
\begin{myeq}\label{eqdergammak}
\hK\sp{\alpha}\sb{\Gamma} \to~\coprod\sb{1\leq i\leq m}~
\hK\sp{\alpha}\sb{\Gamma\setminus\{\gamma\sb{i}\}}\to~\dotsc\to \coprod\sb{1\leq i<j\leq m}\
\hK\sp{\alpha}\sb{\{\gamma\sb{i},\gamma\sb{j}\}}~\xra{\hPsi:=\bot\sb{i<j}\hpsi\sb{i,j}}~
\beta
\end{myeq}
\noindent (see explanation immediately following \wref[),]{eqdergamma}  with $\beta$ of \wref{eqdergammak} identified with corresponding $\beta$ of $\sJ$.
\item The obvious structure maps between the objects described so far corresponding to the morphisms of \w[.]{\Path\sb{\sJ}} 
\item A new structure map \w{j\sb{\alpha}:\hK\sp{\alpha}\sb{\Gamma}\to\alpha} from the initial object in \wref{eqdergammak} to the corresponding object in $\sJ$.
\item Finally, a partial $\Proj$-completion by adding  formal pullbacks of all pairs
\w{(j\sb{\alpha},\theta)} with the same target, where \w{\theta:\alpha'\to\alpha} is any (non-identity) map of $\sJ$. 
\end{enumerate}
\end{defn}

\begin{defn}\label{dderdiag}
Given a $k$-hybrid diagram \w{\fY:\sJ\to\Chp} (see \S \ref{dformal}), its 
\emph{global derived diagram} \w{\fY':\hJD\to\grV} is defined by

\begin{enumerate}
\renewcommand{\labelenumi}{(\alph{enumi})~}
\item Sending the copy of $\sJ$ in \w{\hJD} to \w{\grV} by \w{H\sb{\ast\leq k+1}\fY} (possibly including maps of positive degree).
\item Sending the indexing diagram \wref{eqdergammak} for \w{(\alpha,\Gamma,\beta)} to the partial derived diagram \w{\fY'(\alpha,\beta,\Gamma)} of Definition \ref{dmaxcoll}.
\item Sending the structure maps for the morphisms of \w{\Path\sb{\sJ}} to the central vertical map in \wref[.]{eqsquare}
\item Sending  \w{j\sb{\alpha}:\hK\sp{\alpha}\sb{\Gamma}\to\alpha} to the inclusion  \w[.]{K\sp{\alpha}\sb{\Gamma}\hra H\sb{k}\fY(\alpha)} 
\item Extending in a limit-preserving manner to the partial $\Proj$-completion.
\end{enumerate}
\end{defn}

\begin{lemma}\label{lderform}
The global derived diagram \w{\fY':\hJD\to\grV} of a $k$-hybrid diagram $\fY$ is \wwb{k+1}formal.
\end{lemma}

\begin{defn}\label{dedd}
Given the global derived diagram \w{\fY':\hJD\to\grV} of a $k$-hybrid diagram 
\w[,]{\fY:\sJ\to\Chp} we now construct its \emph{expanded derived diagram} 
\w[,]{\wY=\Exp\fY':\hJIn{2}\to\Chp} starting with \w{\tau\sb{k-1}\wY} equal to \w{\tau\sb{k-1}\wYn{2}} (which is formal); disregarding both, we assume for simplicity that \w[.]{k=0} We shall require $\wY$ to be minimally cofibrant, and present it as a filtered diagram of chain complexes \w[,]{\hF{0}\wY\subseteq \hF{1}\wY\subseteq \hF{2}\wY\dotsc} with
\w[.]{\wY=\bigcup\sb{\ell=0}\sp{N}~\hF{\ell}\wY} The successive subdiagrams are defined inductively\vsm.

\noindent\textbf{I.}\  Let \w{\{\alpha\sp{i}\}\sb{i=1}\sp{n\sb{\init}}=\cF{0}\sJ} (see \S \ref{snac}) and 
\w{\{\omega\sp{j}\}\sb{j=1}\sp{n\sb{\term}}} be the sets of initial and terminal objects of $\sJ$, respectively.

We start with \w{\hF{0}\wY:=H\sb{0}(\fY)\oplus H\sb{1}(\fY)} 
(which is the value of \w{\fY'} on the copy of $\sJ$ inside \w[,]{\hJD} 
by Definitions \ref{ddercat}(a) and \ref{dderdiag}(a)), \ww{\Ind}-completed as in \S \ref{sindcomp}.
In addition, we choose a splitting of \w{H\sb{0}\fY(\alpha\sp{i})} \wb{i=1,\dotsc,n\sb{\init}}  as follows:
\begin{myeq}\label{eqsplitone}
H\sb{0}\fY(\alpha\sp{i})~=~\tilde{H}(\alpha\sp{i})~\oplus~\bigoplus\sb{(\alpha\sp{i},\Gamma)\in\wPath\sb{\sJ}}
\widetilde{K}\sp{\alpha\sp{i}}\sb{\Gamma}
\end{myeq}
\noindent% where $\Gamma$ ranges over subsets of incomparable elements in $\sJ$ 
(see \S \ref{sgdd}), where
\begin{myeq}\label{eqsplittwo}
\widetilde{K}\sp{\alpha\sp{i}}\sb{\Gamma}~:=~K\sp{\alpha\sp{i}}\sb{\Gamma}~\setminus 
\bigcup\sb{\alpha\sp{i}<\delta<\gamma\in\Gamma}~K\sp{\alpha\sp{i}}\sb{(\Gamma\setminus\{\gamma\})\cup\{\delta\}}~
\end{myeq}
\noindent consists of those elements of \w{K\sp{\alpha\sp{i}}\sb{\Gamma}} which vanish for the first time in $\Gamma$, and \w{\tilde{H}(\alpha\sp{i})} consists of those elements of \w{H\sb{0}\fY(\alpha\sp{i})} which survive to all of the homology groups  
\w[.]{H\sb{0}\fY(\omega\sp{j})} The splitting implicit in the notation for \wref{eqsplittwo} is defined by induction on the
filtration of the elements of $\Gamma$\vsm.

\noindent\textbf{II.}\ For the next stages, note that for any \w[,]{\delta\in\sJ\setminus\cF{0}\sJ} we have a sub-vector space \w{H'(\delta)\subseteq H\sb{0}\fY(\delta)} spanned by the images of 
\w{H\sb{0}\fY(\alpha\sp{i})\to H\sb{0}\fY(\delta)} for all 
\w[,]{\alpha\sp{i}\in\cF{0}\sJ} and we choose a splitting for this inclusion (by induction on the filtration of $\delta$, compatibly with \wref[).]{eqsplitone}

Let \w{\tJ:=\sJ\setminus\cF{0}\sJ} with corresponding derived category \w{\tJD} (inside
\w[),]{\hJD} and define a diagram \w{\ttY:\tJD\to\Chp} by setting \w{\ttY(\delta)\sb{0}} to be the chosen complement in \w{\fY'(\delta)\sb{0}} of \w[,]{H'(\delta)} and \w[.]{\ttY(\delta)\sb{1}:=\fY'(\delta)\sb{1}} Note that \w{\cF{s}\tJ=\cF{s+1}\sJ} for all \w[,]{s\geq 0} and by induction we know how to define \w[.]{\hF{\ell-1}\Exp\ttY}

At stage \w[,]{\ell=1} for each \w{\alpha\sp{i}\in\cF{0}} we set \w[,]{\hF{1}\wY(\alpha\sp{i}):=\hF{0}\wY(\alpha\sp{i})} and for each \w{\gamma\in\cF{1}\sJ\setminus\cF{0}\sJ} we set
\begin{myeq}\label{eqfone}
\hF{1}\wY(\gamma):=H'(\gamma)~\oplus~\hF{0}\Exp\ttY(\gamma)~\oplus~\bigoplus\sb{i=1}\sp{n\sb{\init}}\ \bigcup\sb{\omega\sp{j}>\gamma}~
\bigcup\sb{\{\gamma,\delta\} \in\Incomp(\ajb{\alpha\sp{i}}{\omega\sp{j}})}~
CK\sp{\alpha\sp{i}}\sb{\{\gamma,\delta\}}
\end{myeq}
\noindent Observe that \wref{eqfone} is not a direct sum of diagrams: under the morphism \w{\alpha\sp{i}<\gamma} of
$\sJ$, we send those summands of \w{H\sb{0}\fY(\alpha\sp{i})\subseteq \hF{0}\wY(\alpha\sp{i})}
in \wref{eqsplitone} which together constitute  
\w{K\sp{\alpha\sp{i}}\sb{\{\gamma,\delta\}}} to \w{CK\sp{\alpha\sp{i}}\sb{\{\gamma,\delta\}}} in 
\w[.]{\hF{1}\wY(\gamma)} 

Both the coproduct and the unions in \wref{eqfone} are just the appropriate colimits. However, in order to ensure cofibrancy, we must have a coproduct over the various initial objects \w[,]{\alpha\sp{i}} but we have natural inclusions of the cones for \w[,]{\delta<\delta'} say, or the same minimal common join \w{\beta\leq\omega\sp{j},\omega\sp{j'}} 
(the terminal objects only appear to ensure that $\gamma$ and $\delta$ have \emph{some} common join).

For \w{\beta\in\sJ} and every incomparable pair 
\w[,]{\{\gamma,\delta\}\in\Incomp(\ajb{\alpha\sp{i}}{\beta})} there is some 
\w{1\leq j\leq n\sb{\term}} with \w[,]{\beta\leq\omega\sp{j}} and thus inside both \w{\wY(\gamma)} and \w{\wY(\delta)} we have copies of the cone 
\w[,]{CK\sp{\alpha\sp{i}}\sb{\{\gamma,\delta\}}} with inclusions of 
\w{K\sp{\alpha\sp{i}}\sb{\{\gamma,\delta\}}\subseteq H\sb{0}\fY(\alpha\sp{i})}
into each. Therefore, in \w{\hJIn{2}} we have a copy of the pushout 
\w[,]{P\simeq\Sigma K\sp{\alpha\sp{i}}\sb{\{\gamma,\delta\}}} and from Definition
\ref{dderdiag}(b) the derived diagram \w{\fY'} determines the evaluation map 
\w[.]{\psi:P\to H\sb{1}\fY(\beta)\subseteq\wY(\beta)} We now extend the diagram we have obtained to all of \w{\hJIn{2}} by identity maps (so \w{\hF{1}\wY} is cofibrant)\vsm.

\noindent\textbf{III.}\ In the induction stage, assuming we have defined \w{\hF{s}\wY} for all
\w[,]{0\leq s<\ell} we require the following notation:

 For any \w{\gamma,\gamma'\in\sJ} with \w[,]{\alpha\sp{i}<\gamma'<\gamma} we have 
\w[,]{K\sp{\alpha\sp{i}}\sb{\gamma'}\subseteq K\sp{\alpha\sp{i}}\sb{\gamma}}
and we let \w{K'\sb{\gamma}} denote the colimit 
\w{\bigcup\sb{\alpha\sp{i}<\gamma'<\gamma}\ K\sp{\alpha\sp{i}}\sb{\gamma'}} (a partial direct summand of
 \w[,]{K\sp{\alpha\sp{i}}\sb{\gamma}} using \wref[),]{eqsplitone} with 
\begin{myeq}\label{eqkpairs}
\tilde{K}\sp{\alpha\sp{i}}\sb{(\gamma,\delta)}~:=~
(K\sp{\alpha\sp{i}}\sb{\gamma}\setminus K'\sb{\gamma})\cap K\sp{\alpha\sp{i}}\sb{\delta}~.
%\cup~ [K\sp{\alpha\sp{i}}\sb{\gamma}\cap (K\sp{\alpha\sp{i}}\sb{\delta}\setminus K'\sb{\delta})]
%
\end{myeq}
\noindent (the notation \w{K\sp{\alpha\sp{i}}\sb{\gamma}\setminus K'\sb{\gamma}}
depends on the splitting in Step I).

For any \w{\gamma\in\cF{\ell}\sJ\setminus\cF{\ell-1}\sJ} we now define
\begin{myeq}\label{eqftwo}
\hF{\ell}\wY(\gamma)~:=~H'(\gamma)\oplus\hF{\ell-1}\Exp\ttY(\gamma)~ \oplus~\bigoplus\sb{i=1}\sp{n\sb{\init}}\ \bigcup\sb{\omega\sp{j}>\gamma}~
\bigcup\sb{(\gamma,\delta)\in\Incomp(\ajb{\alpha\sp{i}}{\omega\sp{j}})}~
C\tilde{K}\sp{\alpha\sp{i}}\sb{(\gamma,\delta)}~.
\end{myeq}

Again we have evaluation maps from the new parts of each pushout into any common join, dictated by the values in 
\w[,]{\fY'} and again we extend the diagram to all of \w{\hJIn{2}} by identity maps.
\end{defn}

\begin{prop}\label{pexptrun}
Given a $k$-hybrid diagram \w[,]{\fY:\sJ\to\Chp} there is a map \w{\Omega:\wY=\Exp\fY'\to\wYn{2}} of \ww{\hJIn{2}}-indexed diagrams in \w{\Chp} which is  a\wwb{k+1}equivalence, with weak inverse \w[.]{\Phi:\tau\sb{k+1}\wYn{2}\to\wY} 
\end{prop}

\begin{proof}
We construct $\Phi$  by the following procedure, mimicking the construction of \S \ref{dedd} inside the given $\fY$ (assuming for simplicity that \w[,]{k=0} as above). The map $\Omega$ itself will be defined in Step VII at the end of the process. As before, let \w{\{\alpha\sp{i}\}\sb{i=1}\sp{n\sb{\init}}} and 
\w{\{\omega\sp{j}\}\sb{j=1}\sp{n\sb{\term}}} be the initial and terminal objects of $\sI$, respectively\vsm.

\noindent\textbf{Step I:}\ 
Note that \w{H\sb{0}\fY(\alpha)} is the bottom summand of the subdiagram \w{\hF{0}\wY(\alpha)} of \w{\wY(\alpha)} for each  \w{\alpha\in\sJ}  (see \S \ref{dedd}), and since \w{\fY\sb{0}=Z\sb{0}\fY} we use the quotient map \w{q:Z\sb{0}\fY~\to~H\sb{0}\fY} to define 
\w{\Phi\sb{0}:\fY(\gamma)\sb{0}\to H\sb{0}\wY(\gamma)} in \w{\hF{0}\wY(\gamma)\sb{0}\subseteq\wY(\gamma)\sb{0}}
starting with \w[,]{\gamma=\alpha\sp{i}} (where $q$ is an isomorphism), and continuing by naturality for the image
of \w{H\sb{0}\fY(\alpha\sp{i})} in any \w{H\sb{0}\fY(\gamma)} with
\w{\gamma>\alpha\sp{i}} along the unique path from  \w{\alpha\sp{i}} to
any of the terminal  objects \w[\vsm.]{\omega\sp{j}} 

\noindent\textbf{Step II:}\ 
Next, for \w{\alpha<\beta} in $\sJ$, let \w{\oK\sp{\alpha}\sb{\beta}=\oK\sp{\alpha}\sb{\beta}\fY} denote the
subspace of  \w{K\sp{\alpha}\sb{\beta}\subseteq H\sb{0}\fY(\alpha)} consisting of
elements $x$ which go to zero along a \emph{unique} path from $\alpha$ to $\beta$ \wh
that is, if \w{0\neq x\in K\sp{\alpha}\sb{\gamma}\cap K\sp{\alpha}\sb{\delta}} for 
\w[,]{\alpha<\gamma,\delta<\beta} then \w{\delta\leq\gamma} or \w[.]{\gamma\leq\delta}
Let \w[.]{\oK\sp{\alpha}=\oK\sp{\alpha}\fY:=\bigcap\sb{j=1}\sp{n\sb{\term}}\ \oK\sp{\alpha}\sb{\omega\sp{j}}}

The idea is that elements in \w{\oK\sp{\alpha}} cannot contribute to pushout squares of the form \wref[,]{eqsimplesec}
and thus to secondary homotopy invariants of the form \w[;]{d\sp{2}} therefore, we may disregard them, assuming we
have a way to split the cones \w{C\oK\sp{\alpha}} off $\fY$ from the first time they appear, and  send these summands 
to zero in $\wY$\vsm. 

\noindent\textbf{Step III:}\ 
It remains to deal with those elements $x$ which \emph{do} belong to incomparable pairs, and thus may contribute to a 
\ww{d\sp{2}}-differential.  This is done by an outer induction on the filtration of the $\alpha$ where 
\w{0\neq x\in H\sb{0}\fY(\alpha)} first appears. 

For this purpose, we shall require some additional notation (see \S \ref{dedd}\,II above): for any 
\w[,]{\delta\in\sJ\setminus\cF{0}\sJ} we have \w{H'\fY(\delta)\subseteq H\sb{0}\fY(\delta)} spanned by the images of 
\w{H\sb{0}\fY(\alpha\sp{i})\to H\sb{0}\fY(\delta)} for all \w[.]{1\leq i\leq n\sb{\init}} Note that 
non-bounding cycles coming from distinct initial \w{\alpha\sp{i}} and \w{\alpha\sp{j}} may represent the 
same element in  \w[.]{H\sb{0}\fY(\gamma)}
Since \w[,]{H\sb{0}\fY(\alpha\sp{i})=\fY(\alpha\sp{i})\sb{0}} we may choose a splitting for 
\w[.]{H'\fY(\delta)\subseteq\fY(\delta)\sb{0}}

Again, let \w{\tJ:=\sJ\setminus\cF{0}\sJ} and define a diagram \w{\tY:\tJ\to\Chp} by setting \w{\tY(\delta)\sb{0}}
to be the chosen complement of  \w{H'\fY(\delta)} in \w[,]{\fY(\delta)\sb{0}} and 
\w[.]{\tY(\delta)\sb{1}:=\fY(\delta)\sb{1}} 
Again \w{\cF{s}\tJ=\cF{s+1}\sJ} for all \w[,]{s\geq 0} Note that we can extend $\tY$ to all of $\sJ$ by zero,
thus allowing us to repeat this process and define \w{\tY\up{k}:\sJ\to\Chp} inductively to be 
\w{\widetilde{\tY\up{k-1}}:\sJ\to\Chp} (the result of applying the above procedure \w{\fY\mapsto\tY} to \w[,]{\tY\up{k-1}}
extended by zero), with \w[\vsm.]{\tY\up{0}:=\fY}

\noindent\textbf{Step IV:}\ 
We start with
\begin{myeq}\label{eqkalphai}
\cK{0}\fY(\alpha\sp{i})~
:=~\bigcup\sb{(\alpha\sp{i},\Gamma)\in\wPath\sb{\sJ}}\,K\sb{\Gamma}\sp{\alpha\sp{i}}\fY~
\subseteq H\sb{0}\fY(\alpha\sp{i})=\fY(\alpha\sp{i})\sb{0}~.
\end{myeq}
\noindent Evidently \w[,]{\oK\sp{\alpha\sp{i}}\cap \cK{0}\fY(\alpha\sp{i})=0} so we may choose splittings of both of these subspaces
off \w[,]{\fY(\alpha\sp{i})\sb{0}} with complement consisting of those elements which persist in homology to some terminal \w[.]{\omega\sp{j}}
%We do not claim that these splittings are consistent for different initial
%objects \w[.]{\alpha\sp{i}\in\sJ}

%
Since $\fY$ is cofibrant, these summands embed in \w{\fY(\gamma)\sb{0}} for each
\w{(\alpha\sp{i},\Gamma,\beta)\in\Path\sb{\sJ}} with \w[.]{\gamma\in\Gamma} If 
\w[,]{\alpha\sp{i}\neq\alpha\sp{j}} the images of \w{\cK{0}\fY(\alpha\sp{i})} and 
\w{\cK{0}\fY(\alpha\sp{j})} in \w{\fY(\gamma)\sb{0}} are disjoint, for the same reason. We proceed by an inner induction on the filtration of $\gamma$:

If \w[,]{\gamma\in\cF{1}\sJ\setminus\cF{0}\sJ} let 
\begin{myeq}\label{eqdone}
\DDY{1}{\gamma}{0}~:=~\bigoplus\sb{\alpha\sp{i}\leq\gamma}~
\ K\sp{\alpha\sp{i}}\sb{\gamma}\fY\cap \oK\sp{\alpha\sp{i}}\fY~\subseteq~\bigoplus\sb{\alpha\sp{i}\leq\gamma}\,\fY(\alpha\sp{i})\sb{0}~\subseteq~
\fY(\gamma)\sb{0}~,
\end{myeq}
\noindent and
\begin{myeq}\label{eqeone}
\EEY{1}{\gamma}{0}~:=~\bigoplus\sb{\alpha\sp{i}\leq\gamma}~
\bigcup\sb{\omega\sp{j}>\gamma}~\bigcup\sb{\{\gamma,\delta\}\in\Incomp(\ajb{\alpha\sp{i}}{\omega\sp{j}})}~
\ K\sp{\alpha\sp{i}}\sb{\{\gamma,\delta\}}\fY~\subseteq~\bigoplus\sb{\alpha\sp{i}\leq\gamma}\,\fY(\alpha\sp{i})\sb{0}~\subseteq~\fY(\gamma)\sb{0}~.
\end{myeq}
\noindent Again \w[,]{\DDY{1}{\gamma}{0}\cap\EEY{1}{\gamma}{0}=0} and both are coned off in \w[,]{\fY(\gamma)\sb{0}} so we may split their direct sum off the latter, and choose
\begin{enumerate}
\renewcommand{\labelenumi}{(\alph{enumi})~}
\item a splitting \w{\rhu{\gamma}} for the image of \w{d\sb{1}:\fY(\gamma)\sb{1}\to\fY(\gamma)\sb{0}} landing in 
\w[,]{\DDY{1}{\gamma}{0}} with \w{\DDY{1}{\gamma}{1}} defined to be the image of \w{\rhu{\gamma}} (a summand in 
\w[);]{\fY(\gamma)\sb{1}}
\item a splitting \w{\sigu{\gamma}} for the image of\w{d\sb{1}:\fY(\gamma)\sb{1}\to\fY(\gamma)\sb{0}} in 
\w[,]{\EEY{1}{\gamma}{0}} with \w{\EEY{1}{\gamma}{1}} the corresponding summand \w[.]{\fY(\gamma)\sb{1}}
\end{enumerate}

We extend both splittings so as to commute with the inclusions \w{\gamma<\delta} along
any chain from $\gamma$ to \w[,]{\omega\sp{j}} thus defining \w{\DDY{1}{\delta}{\ast}} and \w{\EEY{1}{\delta}{\ast}} for any \w[.]{\delta\in\sJ\setminus\cF{0}\sJ}

We then define $\Phi$ on these two kinds of summands in \w{\fY(\gamma)\sb{\ast}}
by sending the chain complex summand \w{\DDY{1}{\gamma}{\ast}} to zero, and \w{\EEY{1}{\gamma}{\ast}}
isomorphically to the third summand in \wref[,]{eqfone} landing in \w{\hF{1}\wY(\gamma)} of \S \ref{dedd}\vsm.

\noindent\textbf{Step V:}\ 
Note that defining \w{D\fY(\gamma)\sb{0}:=\bigoplus\sb{\alpha\sp{i}\leq\gamma}~
\ K\sp{\alpha\sp{i}}\sb{\gamma}\fY\cap \oK\sp{\alpha\sp{i}}\fY} as in \wref{eqdone} makes sense for any 
\w[,]{\gamma\in\sJ\setminus\cF{0}\sJ} but if \w{\gamma\in\cF{m}\sJ\setminus\cF{0}\sJ} for \w[,]{m>1} it further 
splits by the filtration where the homology classes first vanish (along the unique path from \w{\alpha\sp{i}} via $\gamma$
to \w[)]{\omega\sp{j}} as
\begin{myeq}\label{eqdsplit}
D\fY(\gamma)\sb{0}~=~\bigoplus\sb{s=1}\sp{m}\ \DDY{s}{\gamma}{0}~,\hs\text{a split summand in}\hsm \fY(\gamma)\sb{0}~.
\end{myeq}
\noindent (compare \wref[).]{eqsplittwo}

The summands are defined by induction on \w[,]{m\geq 1} starting with  
\w{\DDY{1}{\gamma}{0}} defined in Step IV (where the first vanishing is in filtration $1$). Assume by induction that \wref{eqdsplit} has been defined for \w{m-1} (and extended to all $\sJ$ under the inclusions \w{\gamma<\delta} as before). Thus for \w{\gamma\in\cF{m}\sJ\setminus\cF{0}\sJ} we have split off all summands but one from 
\w[,]{D\fY(\gamma)\sb{0}} and the remainder \w{\DDY{m}{\gamma}{0}}  consists by definition of those elements in \w{K\sp{\alpha\sp{i}}\sb{\gamma}\fY\cap \oK\sp{\alpha\sp{i}}\fY} \wb{i=1,\dotsc n\sb{\init}} which did not vanish in any earlier filtration.  As before, we can choose splittings for \w{d\sb{1}} to define (contractible) chain complex 
summands \w{ \DDY{s}{\gamma}{\ast}} \wb{1\leq s\leq m} in \w[.]{\fY(\gamma)}

Similarly, we can use  \wref{eqeone} to define \w{E\fY(\gamma)\sb{0}} for any 
\w{\gamma\in\cF{m}\sJ\setminus\cF{0}\sJ} \wb[,]{m>1} and can split it into a direct sum by the filtration of first vanishing:
\begin{myeq}\label{eqesplit}
E\fY(\gamma)\sb{0}~=~\bigoplus\sb{s=1}\sp{m}\ \EEY{s}{\gamma}{0}~,
\end{myeq}
\noindent itself a split summand in \w[.]{\fY(\gamma)\sb{0}} The summand \w{\EEY{s}{\gamma}{0}} is defined by induction, starting with \w{\EEY{1}{\gamma}{0}} of Step IV. Assume that \wref{eqesplit} has been defined for \w{m-1} and extended to all $\sJ$; then for \w{\gamma\in\cF{m}\sJ\setminus\cF{0}\sJ} we have 
split off all summands but one from  \w[,]{E\fY(\gamma)\sb{0}} and the remainder \w{\EEY{m}{\gamma}{0}}  consists by definition of those elements in \w{K\sp{\alpha\sp{i}}\sb{\{\gamma ,\delta\}}} \wb{i=1,\dotsc n\sb{\init}} which did not vanish for any 
\w[.]{\gamma'<\gamma}  As before, we can choose splittings for \w{d\sb{1}} to define (contractible) chain complex 
summands \w{ \EEY{s}{\gamma}{\ast}} \wb{1\leq s\leq m} in \w[\vsm.]{\fY(\gamma)}

\noindent\textbf{Step VI:}\ 
The construction of \w{\Phi} now proceeds by induction on the filtration length of
\w{\alpha\in\sJ} from  the initial objects
\w[,]{\cF{0}\sJ=\{\alpha\sp{i}\}\sb{i=1}\sp{n\sb{\init}}} where for each $\gamma$ in filtration \w{\cF{m}\sJ\setminus\cF{m-1}\sJ} \wb[,]{m\geq 2} we have chosen split summands
\begin{myeq}\label{eqsplitsum}
\bigoplus\sb{s=1}\sp{m}\ \left[ H'\tY\up{s-1}(\gamma)~\oplus~D\tY\up{s-1}(\gamma)\sb{\ast}~\oplus~
E\tY\up{s-1}(\gamma)\sb{\ast}\right]
\end{myeq}
\noindent in \w{\fY(\gamma)} (in the notation of Steps II-IV). The summands \w{D\tY\up{s-1}(\gamma)\sb{\ast}} are 
sent to $0$ under $\Phi$, the summands \w{E\tY\up{s-1}(\gamma)\sb{\ast}} are sent isomorphically to the third summand in \wref{eqftwo} (in each filtration), and the $0$-cycles were already sent to the \w{H\sb{0}\fY} summands in $\wY$ in Step I. Having split off all the cones elements in \w[,]{\fY(\gamma)\sb{1}} we are left only with the $1$-cycles (modulo the boundaries \wh see \wref[),]{eqtrunc} which we  send to isomorphically to the \w{H\sb{1}\fY} summands in $\wY$\vsm. 

\noindent\textbf{Step VII:}\ 
The resulting map \w{\Phi:\tau\sb{1}\wYn{2}\to\wY} has a section
\w[,]{\Omega':\wY\to\tau\sb{1}\wYn{2}} which is an inclusion into the appropriate summands (and thus still a weak equivalence).
This extends to a natural transformation \w{\Omega:\wY\to\wYn{2}} of \ww{\hJIn{2}}-diagrams, because $\wY$ is zero above dimension $1$. Since it is an isomorphism on \w{\hJIn{2}} away from the images of \w[,]{(\sigma\sb{\gamma},\rho\sb{\gamma})} which are inclusions into cones, and thus do not contribute to the weak homotopy type of \w[,]{\fY(\alpha)} the map $\Omega$ is a $1$-equivalence (in the general case, a \wwb{k+1}equivalence).
\end{proof}

\begin{corollary}\label{cexptrun}
The  \wwb{k+1}truncation of a $k$-hybrid diagram \w{\fY:\sJ\to\Chp} 
may be recovered up to weak equivalence from its global derived diagram 
\w[.]{\fY':\hJD\to\grV} 
\end{corollary}

\begin{defn}\label{dhybrdztn}
Assume given a $k$-hybrid diagram \w{\fY:\sJ\to\Chp} and its global derived diagram \w{\fY':\hJD\to\grV\subseteq\Chp} (a \wwb{k+1}hybrid \wwb{k+1}-truncated diagram). 
We enlarge the derived category \w{\hJD} to a category \w{\hJH} which also includes 
\w{\hJIn{2}} as a subcategory, by adding structure maps for all (formal) colimits 
\w{\hK\sp{\alpha}\sb{\Gamma'}} of \wref[.]{eqdergammak}
We then extend \w{\fY'} to a functor \w{\fY'':\hJH\to\grV\subseteq\Chp} by giving the value $0$ to these structure maps. Similarly, we extend \w{(\wYn{2})\lra{k+1}:\hJIn{2}\to\Chp} (see \S \ref{sindcomp}) \ -- \ by zero, where undefined \ -- \ to a functor \w[.]{\wYu:\hJH\to\Chp} 

We have another diagram \w[,]{\fZ:\hJH\to\Chp}
taking the maps \w{\hPsi} and their restrictions to the various \w{\hK\sp{\alpha}\sb{\Gamma'}} in \wref[,]{eqdergammak} for all \w{(\alpha,\Gamma,\beta)} in \w{\Path\sb{\sJ}} and all \w{\Gamma'\subsetneq\Gamma} to the corresponding induced maps in \w{Z\sb{k+1}\wYn{2}} (again extending by zero, where undefined).
Since the inclusion 
\w{\inc:Z\sb{k+1}\wYn{2}(\alpha)\hra (\wYn{2})\lra{k+1}(\alpha)} is a map in \w{\Chp}
for each \w[,]{\alpha\in\hJIn{2}} we have a map of diagrams \w[,]{\fZ\to\wYu}
and the quotient maps\w{q\sb{k+1}:Z\sb{k+1}\wYn{2}\to H\sb{k+1}\wYn{2}}
induce another map of \ww{\hJIn{2}}-diagrams \w[.]{\fZ\to\fY'}

The \emph{\wwb{k+1}hybrid approximation} \w{\oHyb{k+1}\fY:\hJH\to\Chp} of $\fY$ is defined to be the pushout:
$$
 \xymatrix
 {
   \fZ \ar[rr] \ar[d]  && \fY'' \ar[d] \\
   \wYu\lra{k+1} \ar[rr] && 
   \oHyb{k+1}\fY 
 }
$$
\noindent The pushout assigns to the maps \w{\hPsi} and the induced maps \w[,]{\hpsi\sb{i,j}} and so on, in \wref{eqdergammak} the corresponding maps \w{\psi\sb{i,j}} of \wref[,]{eqdergamma} which also occur 
in \w[.]{(\wYn{2})\lra{k+1}}
Similarly for the maps of \w[,]{H\sb{k+1}\fY} which occur in 
\S \ref{dderdiag}(a) and are
\w{[\fY\lra{n}(\alpha)]\sb{n+1}=Z\sb{n+1}\fY(\alpha)\epic H\sb{n+1}\fY(\alpha)} for each $\alpha$ in \wref[,]{eqconnc} as in the proof of Proposition \ref{pexptrun}.

Note that the \wwb{k+1}hybrid approximation \w{\oHyb{k+1}\fY} to $\fY$ is indeed \wwb{k+1}hybrid, since the two diagrams \w{\fY'} and \w{(\wYn{2})\lra{k+1}} used in the pushout are such.
\end{defn}

\begin{thm}\label{thybrdztn}
A $k$-hybrid diagram $\fY$ may be recovered up to weak equivalence from its \wwb{k+1}hybrid approximation \w[.]{\oHyb{k+1}\fY}
\end{thm}

\begin{proof}
Applying the procedure described in \S \ref{dedd}, we may expand the pushout diagram \w{\oHyb{k+1}\fY} 
(restricted to \w[)]{\hJIn{2}} into the pushout diagram
$$
 \xymatrix
 {
   \fZ \ar[rr] \ar[d]  && \wY=\Exp\fY' \ar[d] \ar@/^2ex/[rrdd]\sp{\Omega} \\
   \wYn{2}\lra{k+1} \ar@/_2ex/[rrrrd]\sb{\iota\sb{k+1}} \ar[rr] && 
   \tY \ar@{.>}[rrd]\sp{\Psi}\\
&& && \wYn{2}
 }
$$
\noindent in the category of diagrams from \w{\hJIn{2}} to \w[.]{\Chp}

The \wwb{k+1}equivalence $\Omega$ of Proposition \ref{pexptrun} is zero in degrees \w[,]{>k+1} while the covering map \w{\iota\sb{k+1}} (see \S \ref{dtrunc}) is zero in degrees \w[,]{<k+1} and an isomorphism in degrees \w[,]{>k+1} so we need only verify the compatibility of the two maps in degree \w[.]{k+1} 

Given $m$ incomparable objects \w{\Gamma=\{\gamma\sb{i}\}\sb{i=1}\sp{m}} in \w{\acj{\alpha}} with common join $\beta$, from the description in 
\S \ref{smulpr} we see that the chain complex \w{C:=\hocolim\sb{\ajb{\alpha}{\Gamma}}\fY} has \w{Z\sb{k+1}C=H\sb{k+1}C} canonically isomorphic
to the colimit of \wref{eqdergamma} (omitting the rightmost term) \wh which is precisely the value of \w{\fY'} at \w{(\alpha,\Gamma)} in 
\w[.]{\wPath\sb{\sJ}} This is also true for all \w[,]{\Gamma'\subset\Gamma} so \w{\fY'} (or \w$\wY$) and \w{(\wYn{2})\lra{k+1}} agree on the
maps $\hPsi$ and the induced maps \w[,]{\hpsi\sb{i,j}} etc., of $\fZ$.

Note that \w{(\wYn{2})\lra{k+1}} is itself \wwb{k+1}hybrid (see \S \ref{dformal}), so we may apply the construction of $\Phi$ in the proof of Proposition \ref{pexptrun} to \w{(\wYn{2})\lra{k+1}} in dimension \w{k+1} to obtain splittings for 
the \wwb{k+1}structure maps 
$$
q\sb{k+1}:\wYn{2}\lra{k+1}(\alpha)\sb{k+1}=Z\sb{k+1}\wYn{2}\lra{k+1}(\alpha)~\to~ H\sb{k+1}\wYn{2}(\alpha)
$$
\noindent for all \w[.]{\alpha\in\hJIn{2}} These splittings guarantee that the copies of \w{H\sb{k+1}\wYn{2}(\alpha)} in \w{\wYn{2}\lra{k+1}}
and \w{\wY} are aligned (via $\fZ$), so they map identically to 
\w[,]{\wYn{2}} guaranteeing that $\Psi$ is indeed a weak equivalence in dimension \w[,]{k+1} too.
\end{proof}

\begin{defn}\label{dhybrd}
Theorem \ref{thybrdztn} allows us to define the \emph{$k$-hybridization} \w{\Hyb{k}\fX:\sIu{k}\to\Chp} of an arbitrary (minimally cofibrant) diagram \w[,]{\fX:\sI\to\Ch} as outlined in \S \ref{soas}, by induction on \w[,]{k\geq 0} starting with \w[,]{\Hyb{0}\fX:=\fX}
by setting \w{\Hyb{k}\fX:=\oHyb{k}(\Hyb{k-1}\fX)} and \w{\sIu{k}:=(\sIu{k-1})\sb{\Hybr}} for \w[.]{k>0}

 It also allows us to define the \emph{$k$-th derived diagram} \w{\Der{k}\fX:\hIDu{k}\to\grV} of an arbitrary (minimally cofibrant) diagram \w{\fX:\sI\to\Ch} by induction on \w[,]{k\geq 0} starting with \w[,]{\Der{1}\fX:=\fX'}
by setting \w{\Der{k}\fX:=(\Hyb{k-1}\fX)'} for \w[.]{k>1}
\end{defn}

\begin{remark}\label{rfibrant}
In our discussion so far we have made much use of the fact that all diagrams were made cofibrant (in the model category structure of \cite[\S 11.6]{HirM}, although we did not describe it explicitly). It is of course possible to make them fibrant instead: the only difference this would make in our constructions would be to reverse the order of the formal differentials and the evaluation map in Definition \ref{dmaxcoll} (with the value unchanged). In addition, in the proof of Proposition \ref{pexptrun} we would need to carry out our inductive process from terminal to initial objects in $\sI$. Note that if \w{\fX(\alpha)} is of finite type for each \w[,]{\alpha\in\sI} we can pass between the two versions by taking vector space duals of the diagrams (with the resulting cochain complexes then non-positively graded). As we shall see in Section \ref{css} below, there are sometimes technical advantages to preferring the fibrant version.

An approach which is perhaps preferable conceptually is to work in the ``linear $\infty$-category'' of
\w{\Ch} enriched in (unbounded) chain complexes. However, this would make the point-set descriptions used throughout this paper needlessly complicated.
\end{remark}
 %
 %c5      Examples of higher order homotopy invariants
 %
 \sect{Examples of higher order homotopy invariants}
 \label{cetoi}
 
 We now apply our methods to describe the complete set of higher order homotopy invariants (in graded vector spaces) for the weak homotopy type of two simple but important infinite sequences of diagrams in chain complexes. 

\begin{mysubsection}{Cubes}\label{mechcube}
\label{scube}
 Let \w{\sI\sp{n}} denote the lattice of subsets of \w[,]{[n] = \{1, 2, \cdots , n\}} which we shall think of an $n$-dimensional cube with vertices labelled by 
 \w{J=(\epsilon_1, \epsilon_2, \cdots, \epsilon_n)} with each \w{\epsilon_i \in \{0,1\}} 
 (so \w{J} is the characteristic function of a subset of \w[).]{[n]} 
 We denote by \w{J_A} the vertex associated  to \w[,]{A\subseteq[n]} and we have a unique map  \w{J\sb{B}\to J\sb{A}} in \w{\sI\sp{n}} exactly when \w{A\subseteq B} (the reason for reversing the direction will become apparent in \S \ref{sdss} below).

 For a fixed \ww{\bF}-vector space \w{V} define the diagrams \w{\fD^n_V:  \sI^n \to \Ch}  by 
 $$
\fD^n_V(J\sb{A})=
\begin{cases}
    K(V, 0) & \text{if}\ A=[n]\\
    K(V, n-1) & \text{if}\ A=\emptyset\\
    C^{n-|A|} K(V,0) & \text{otherwise},
\end{cases}
$$
\noindent where \w{C\sp{k}} is the appropriate cone (see \S \ref{snac}) and \w{|A|} denotes the cardinality of $A$.
The maps are determined by requiring that the diagram 
 \w{(\fD^n_V)\rest{\sI^n\setminus \vec{0}}} (omitting the last vertex \w[)]{\vec{0}:=(0, \cdots, 0)} be minimally cofibrant. 
 Note that the homotopy colimit for this sub-diagram is weakly equivalent to \w[,]{K(V, n-1)} and the weak homotopy type of \w{\fD^n_V} is determined by requiring that the structure map
 $$
 \psi\colon\hocolim_{\sI^n\setminus \vec{0}}\ \fD^n_V~\to~\fD^n_V(\vec{0})
 $$
be a weak equivalence.
For \w[,]{A \subseteq [n]} we denote the kernel  of the fold map by  
\w[,]{Z(A):=\ker(\bigoplus_{i\in A} V_i \stackrel{\nabla}{\to}  V)} where \w{V\sb{i}} is a copy of $V$. 
\end{mysubsection}

\begin{mysubsection}{The $3$-dimensional cube}
\label{sthreecube}
We describe in detail the full analysis of \w[:]{\fX=\fD^3_V} a (minimally) cofibrant version appears as the solid diagram in \wref{eqexpcubea} (with the relevant part of the \ww{\Ind}-completion \w{\wYn{2}} indicated by boxes for the colimits, with structure maps indicated by dashed arrows):
{\footnotesize
\mypdiag[\label{eqexpcubea}]{
    K(V,0) \ar[ddrr] \ar[ddddd] \ar[rrrr]& & & & 
    \binom{V_1}{V} \ar@{-->}[dddl]^-{}|(0.495)\hole \ar@{-->}[dl] \ar[ddddd]^-{}|(0.27)\hole^-{}|(0.4)\hole \ar[ddrr] & & \\
     & & & \fbox{$\binom{V\sb{1}\oplus V\sb{2}}{V}$} \ar@{-->}[drrr] & & \\
     & & \binom{V_2}{V} \ar@{-->}@<-1ex>[ddl] \ar@{-->}[ur] \ar[ddddd] \ar[rrrr]  & & & & 
     {\begin{psmallmatrix} V_{12}\\
      V_1\oplus V_2\\ V
  \end{psmallmatrix}} \ar[ddddd] \\
     & & & \fbox{$\binom{V_1\oplus V_3}{V}$} \ar@{-->}[ddr] & & &\\
     & \fbox{$\binom{V_2\oplus V_3}{V}$} \ar@{-->}[dddr] \\
    \binom{V_3}{V} \ar@{-->}[ur] \ar@{-->}@<-1ex>[uurrr]^-{}|(0.495)\hole \ar[ddrr]  \ar[rrrr]^-{}|(0.335)\hole^-{}|(0.27)\hole &  & & & {\begin{psmallmatrix} V_{13}\\
      V_1\oplus V_3\\ V
  \end{psmallmatrix}} \ar[ddrr]  & & \\
 & & & & & & \\
     & & {\begin{psmallmatrix} V_{23}\\
      V_2\oplus V_3\\ V
  \end{psmallmatrix}} \ar[rrrr] & & & &  {\begin{psmallmatrix} V_{12} \oplus V_{13} \oplus V_{23}\\
      V_1 \oplus V_2 \oplus V_3\\ V
  \end{psmallmatrix}}\simeq K(V,2)
  }}
Note that if \w{\fD^3_V} is (minimally) cofibrant, we may assume that \w{\fD^3_V(\vec{0})} is 
\w[,]{\hocolim_{\sI^3\setminus \vec{0}}\ \fD^3_V} as indicated in \wref[.]{eqexpcubea}

The $0$-connected cover \w{\wYn{2}\langle 1 \rangle} is shown in \wref[:]{eqexpcubeb}

{\footnotesize
\mypdiag[\label{eqexpcubeb}]{
    0 \ar[ddrr] \ar[ddddd] \ar[rrrr]& & & & 
    0 \ar@{-->}[dddl]^-{}|(0.495)\hole \ar@{-->}[dl] \ar[ddddd]^-{}|(0.23)\hole^-{}|(0.37)\hole \ar[ddrr] & & \\
     & & & \binom{Z(1,2)}{0} \ar@{-->}[drrr] & & \\
     & & 0 \ar@{-->}@<-1ex>[ddl] \ar@{-->}[ur] \ar[ddddd] \ar[rrrr]  & & & & 
     {\begin{psmallmatrix} V_{12}\\
      Z(1,2)\\ 0
  \end{psmallmatrix}} \ar[ddddd] \\
     & & & {\binom{Z(1,3)}{0}} \ar@{-->}[ddr] & & &\\
     & \binom{Z(2,3)}{0} \ar@{-->}[dddr] \\
    0 \ar@{-->}[ur] \ar@{-->}@<-1ex>[uurrr]^-{}|(0.495)\hole \ar[ddrr]  \ar[rrrr]^-{}|(0.4)\hole^-{}|(0.27)\hole &  & & & {\begin{psmallmatrix} V_{13}\\
      Z(1,3)\\ 0
  \end{psmallmatrix}} \ar[ddrr]  & & \\
 & & & & & & \\
     & & {\begin{psmallmatrix} V_{23}\\
      Z(2,3)\\ 0
  \end{psmallmatrix}} \ar[rrrr] & & & &  {\begin{psmallmatrix} V_{12} \oplus V_{13} \oplus V_{23}\\
      Z(1,2,3)\\ 0
  \end{psmallmatrix}}
}}

The $1$-derived diagram \w{\Der{1}\fX= \fX'} is shown in \wref{eqexpcubec}
(where we have omitted the irrelevant primary information). The maps from \w{V} to \w{Z(i,j)} (for all \w[)]{1\le i < j \le 3} are of degree one.

{\footnotesize
\mypdiag[\label{eqexpcubec}]{
    V \ar[drrr] \ar[ddddr] \ar[ddddrrr]^-{}|(0.67)\hole \ar[ddrr] \ar[ddddd] \ar[rrrr]& & & & 
    0  \ar[ddddd]^-{}|(0.27)\hole^-{}|(0.4)\hole \ar[ddrrrr] & & \\
     & & & Z(1,2) \ar@{-->}[drrrrr] & & \\
     & & 0 \ar[ddddd] \ar[rrrrrr]  & & & & & &
     0 \ar[ddddd] \\ &&&&&&\\
     & Z(2,3) \ar@{-->}[dddr]  & & Z(1,3) \ar@{-->}[dr]\\
    0  \ar[ddrr]  \ar[rrrr]^-{}|(0.335)\hole^-{}|(0.27)\hole &  & & & 0 \ar[ddrrrr]  & & \\
 & & & & & & \\
     & & 0 \ar[rrrrrr] & & & & & & 0
  }}

The hybridization \w{\oHyb{1} \fX = \Hyb{1} \fX} is shown in \wref[,]{eqexpcubee}
with the dashed arrows indicating the formal (primary) information (the vertex labeled by the boxed \w{\hocolim} should be added for the $\Ind$-completion). 
 {\footnotesize
\mypdiag[\label{eqexpcubee}]{
K(V,0) \ar[drrr]\sb{\hdt{1,2}} \ar[ddddr]\sp{\hdt{1,3}} 
\ar[ddddrrr]^-{}|(0.67)\hole\sb{\hdt{2,3}} \ar@{-->}[ddrr] \ar@{-->}[ddddd] \ar@{-->}[rrrr]& & & & 
0  \ar@{-->}[ddddd]^-{}|(0.273)\hole^-{}|(0.435)\hole \ar[ddrrrr] & & \\
& & & \binom{Z(1,2)}{0} \ar[drrrrr]\sp{\psi\sb{1,2}} & & \\
     & & 0 \ar@{-->}[ddddd] \ar@{-->}[rrrrrr]  & & & & & &
\mbox{$\begin{psmallmatrix} V\sb{12}\\ Z(1,2)\\ 0\end{psmallmatrix}$} 
\ar[ddddd] \ar[ddddddr] \\ &&&&&&\\
& \binom{Z(2,3)}{0} \ar[dddr]\sp(0.3){\psi\sb{2,3}}  & & 
\binom{Z(1,3)}{0} \ar[dr]\sp{\psi\sb{1,3}}\\
0  \ar@{-->}[ddrr]  \ar@{-->}[rrrr]^-{}|(0.35)\hole &  & &&
\mbox{$\begin{psmallmatrix} V\sb{13}\\ Z(1,3)\\ 0 \end{psmallmatrix}$} 
\ar@{-->}[ddrrrr]  & & \\
& & & & & & \\
& & 
\mbox{$\begin{psmallmatrix} V\sb{23}\\ Z(2,3)\\ 0\end{psmallmatrix}$}
\ar[drrrrrrr] \ar[rrrrrr] & & & & & & \fbox{$\hocolim$} \ar[dr]\sp{=}\\ 
&&&&&&&&& 
\mbox{$\begin{psmallmatrix} 
V\sb{12} \oplus V\sb{13} \oplus V\sb{23}\\ Z(1,2,3)\\ 0
\end{psmallmatrix}$}
}}
\noindent Here the values \w[,]{\psi\sb{1,2}}  \w[,]{\psi\sb{1,3}} and  \w{\psi\sb{2,3}}
are necessarily zero, because they land in cones, so the secondary data for the original $\fX$ all vanishes.

In the next step, starting with \w[,]{\fY:=\Hyb{1} \fX} we repeat the above steps to obtain the third-order information \wh that is, the $2$-derived diagram 
\w{\fY'= (\Hyb{1} \fX )'=\Der{2} \fX} for \w[.]{\fX=\fD^3_V} The relevant part of \w{\fY'} itself consists of the diagram $\fY$ over \w{\beta=\vec{0}} (the last vertex of the cube), shown in \wref[:]{eqexpcubef}  

\myvdiag[\label{eqexpcubef}]{
& & \tS V(1)  \ar[dr]\sp{\hat{\psi}\sb{1}} \\
V \ar[rr]\sp(0.6){\hdt{2}} \ar[urr]\sp(0.4){\hdt{1}} \ar[drr]\sb(0.4){\hdt{3}} & & 
\tS V(2)  \ar[r]\sp{\hat{\psi}\sb{2}} & H_1(\fY(\beta))=0\\ 
  &&  \tS V(3)  \ar[ur]\sb{\hat{\psi}\sb{3}}  
  }
\noindent Note that by the definition of the derived diagram \w[,]{\fY'} there is a copy of \w{\tS V} for each pair 
\w{\{Z(i,j),\,Z(i,k)\}} of facets adjacent to the initial vertex in \wref[.]{eqexpcubee} We denote this copy succinctly by \w{\tS V(i)} 
(where \w{\{i\}} is the intersection of the two pairs).
\end{mysubsection}

\begin{mysubsection}{The third order differential}
\label{sthreediff}
Note that even though the values of the three differentials in \wref{eqexpcubef} are necessarily zero, they are \emph{not} the same as those shown in \wref[,]{eqexpcubee} which landed in three different targets. Because we have a common target \w[,]{\beta=\vec{0}} 
these three values are coordinated, because the homotopy colimit of the left half of \wref{eqexpcubef} (in graded vector spaces) is isomorphic non-canonically to \emph{two} copies of \w{\Sigma V} (see \S \ref{sdhc} above), say 
\w[.]{\Sigma V\sb{I}\amalg \Sigma V\sb{II}}
Thus in order obtain the derived diagram \w{\fY'=\Der{2}\fD^3_V} from \wref{eqexpcubef} we need to 
determine the two values \w{\hat{\psi}\sb{I}} and \w{\hat{\psi}\sb{II}} from \w{\Sigma V} to \w[,]{ H_1(\fY(\beta))=0} which are of course zero.

Thus we are in the graded situation of \S \ref{egdifftwo} (but with maps of positive degree). Therefore, 
(the non-primary part of) \w{\fY'} consists of a single 
differential \w{\dt{1,2,3}:V\to H_{2}(\fD^3_V(\beta))=V} (of degree two). In our case, this is just the isomorphism in the lower right corner of \wref[,]{eqexpcubee} by construction.

Note that this isomorphism represents the value of a third order operation \wh in fact, of a \ww{d\sp{3}}-differential, in the case of a spectral sequence (see \S \ref{sdss} below). It is one of the advantages of our approach that all such higher order information is reduced to secondary structure (in a more complicated diagram), and from there to a description in graded vector spaces alone. 
\end{mysubsection}

\begin{mysubsection}{A sequence of minimal diagrams}
\label{smindiag}
The example of the three-dimensional cube has analogues in all dimensions, providing examples of 
$r$-differentials for all \w[.]{r\geq 2} However, describing these in detail necessitates introducing cumbersome new notation (see \cite{BBSenT}), so instead we now present a more economical sequence of diagrams displaying higher order information of all orders. These diagrams have the following simple form in the homotopy category in 
\w[:]{\ho\Ch\approx\grV}
\myqdiag[\label{eqexpminalh}]{
& 0 \ar[ddrr] \ar[rr] && 0 \ar[r] & \cdots \ar[r] & 0 \ar[ddrr] \ar[rr] && 0 \ar[dr] \\ 
K(V,0) \ar[dr] \ar[ur]  & & &&  &  &&& K(V,n) \\ 
& 0 \ar[uurr]|\hole \ar[rr] & & 0 \ar[r] & \cdots \ar[r] & 0 \ar[uurr]|\hole \ar[rr] \ar[rr] && 0 \ar[ur]}
\noindent (for a fixed vector space $V$). This may be thought of as a stripped-down version of the \wwb{n+1}cube diagram of \S \ref{scube}, with identical higher order information. 

The following diagram \w{\fM{n}} is a cofibrant version in \w{\Ch} of diagram \wref{eqexpminalh} (except for the fact that we may choose a minimal model of the final object 
\w[):]{K(V,n)}
\myqdiag[\label{eqcofm}]{
& CL^0 \ar[ddrr] \ar[rr] && CL^1 \ar[r] & \cdots \ar[r] & CL^{n-2} \ar[ddrr] \ar[rr] && CL^{n-1} \ar[dr] \\ K(V,0) \ar[dr] \ar[ur]  & & &&  &  &&& K(V,n) \\  & C'L^0 \ar[uurr]|\hole \ar[rr] & & C'L^1 \ar[r] & \cdots \ar[r] & C'L^{n-2} \ar[uurr]|\hole \ar[rr] \ar[rr] && C'L^{n-1} \ar[ur]
}
\noindent The various cones (representing the zeros of \wref[)]{eqexpminalh} are given explicitly as follows:
let \w{V'} be an isomorphic copy of $V$. For each \w[,]{k\geq 0} we have a model \w{L\sp{k}} for \w{K(V,k)} with \w{(L\sp{k})\sb{i}=V(i) \oplus V'(i)} (two isomorphic copies of $V$) for \w{1 \le i \le k} and 
\w{(L\sp{k})\sb{0}=V} (and the obvious differentials). Thus \w{CL^{k}} has a new cone cell \w{V(k+1)} in degree \w[,]{k+1} with the isomorphic cone \w{C'L^{k}}  having \w{V'(k+1)} instead of \w[.]{V(k+1)} 

Assuming \w[,]{n>1} the global derived diagram \w{\Der{1} \fM{n}} is given by
\mywdiag[\label{eqderonem}]{
& \Sigma V\ar[r]\sp{\psi}  & 0 \ar[ddr]\ar[r] & 0 \cdots \ar[r] & 0\ar[ddr] \ar[r] & 
0 \ar[dr] \\ 
V \ar[ur]^{\hdtw} \ar[urr]\sb(0.6){d\sp{2}} \ar[dr]_{\hdtw} \ar[drr]\sp(0.6){d\sp{2}}
&&&&&& H_{1}K(V,n)=0~. \\ 
 & \Sigma V' \ar[r]\sb{\psi'} & 0  \ar[uur] \ar[r] & 0 \cdots \ar[r] & 
 0\ar[uur]  \ar[r] & 0 \ar[ur] &
}
\noindent We have one copy of \w{\Sigma V} for each of the two targets of the formal differentials \w{\hdtw} 
(degree $1$ maps in \w[),]{\grV} with evaluation maps $\psi$ and \w{\psi'} (see \S \ref{dmaxcoll}) necessarily trivial (assuming \w[).]{n\geq 2} 

The $0$-connected cover of the \ww{\Ind}-completion \w{(\fM{n})^{[2]}\langle 1\rangle} is given
by:
$$
\xymatrix@R=10pt@C=3pt{
&  \fbox{$\tS L\sp{0}$} \ar[rr] && \tS CL^0 \ar@{-->}[dr]  \ar[r] & \cdots \ar[r] & 
\tS CL^{n-3}\ar@{-->}[dr]  \ar@/^5ex/[ddrr] \ar[rr] && \tS CL^{n-2} \ar@{-->}[dr] \ar@/^2ex/[drrr] \\ 
0 \ar[dr] \ar[ur]  &&   & & \fbox{$\tS L^1$} &    &  \fbox{$\tS L^{n-2}$}  \ar@{-->}[ur] \ar@{-->}[dr] &  & \fbox{$\tS L^{n-1}$} \ar@{-->}[rr]&& K(V,n)\\
&  \fbox{$\tS L\sp{0}$}  \ar[rr] & & \tS C'L^0  \ar@{-->}[ur]  \ar[r] & \cdots \ar[r] & \tS C'L^{n-3} \ar@{-->}[ur] \ar@/_5ex/[uurr]|(0.58)\hole|(0.74)\hole \ar[rr] \ar[rr] && \tS C'L^{n-2} \ar@{-->}[ur] \ar@/_2ex/[urrr]
}
$$
\noindent where $\tS$ is the reduced suspension (see \S \ref{snac}), and as before the new (homotopy) colimits are boxed.  The two copies of \w{\Sigma V} in \w{\Der{1}\fM{n}} re-appear here in the form of two copies of 
\w[,]{\tS L^0} as part of the \ww{\Ind}-completion. The intermediate pushouts \w[,]{\tS L^1}  \w[,]{\tS L^2} and so on should also be doubled, but as they are not needed at this stage we write each only once to avoid crowding the diagram.

The hybridization \w{\Hyb{1}\fM{n}} is given by:
\footnotesize{
\myxdiag[\label{eqhybonem}]{
&& \Sigma V\ar[rr]\sp(0.4){\psi} && \tS CL^0 \ar[dr]  \ar[rr] && \cdots \sh{}(CL^{n-3})\ar[dr]  \ar@/^4ex/[ddrr] \ar[rr] && \sh{}(CL^{n-2}) \ar[dr] \ar@/^2ex/[drrr] \\ 
V \ar[rru]^{\hdtw\sb{1}} \ar[rrd]_{\hdtw\sb{1}} %\ar@{.>}[rrrrr]\sp{\hdtw\sb{2}} 
&&&&&   \tS L^1 && \tS L^{n-2}  \ar[ur] \ar[dr] &  & \tS L^{n-1} \ar[rr] && K(V,n)\\
&&  \Sigma V'\ar[rr]\sb(0.4){\psi'} && \tS C'L^0  \ar[ur]  \ar[rr] && \cdots \sh{}(C'L^{n-3}) \ar[ur] \ar@/_4ex/[uurr]|(0.58)\hole|(0.7)\hole \ar[rr] \ar[rr] && \sh{}(C'L^{n-2}) \ar[ur] \ar@/_2ex/[urrr]
}
}
\normalsize{}

Note that the diagram \w{\Hyb{1}\fM{n}} is isomorphic to \w{\tS \fM{n-1}} (that is, \w{\fM{n-1}} shifted up one degree), with a  formal differential 
\w{\hdtw\sb{1}:K(V,0)\to \tS L^0=K(V,1)} of degree $1$
adjoined at the beginning (together with its factorization through two (trivial)
differentials of the form \w[,]{\psi\circ\hdtw:K(V,0)\to \tS CL^0} and their structure maps to \w[).]{\tS L^1}

By induction on \w[,]{n\geq 2} we find that \w{\Hyb{k} \fM{n}} is given by 
\footnotesize{
\myxdiag[\label{eqhybkm}]{
&&&&&& \sh{k} CL\sp{0} \ar[dr]  \ar[rr] && \cdots \sh{k} CL\sp{n-k-1} \ar[dr] \ar@/^2ex/[drrr] \\
V \ar[rr]\sp-{\hdtw_1} &&\cdots \ar[rr]\sp-{\hdtw_{k}} && \sh{k}V  \ar[urr] \ar[drr]  & && \sh{k} L^1 &    &   \sh{k} L\sp{n-k} \ar[rr]\sp{\psi} && K(V,n)~.\\ 
&&&&&& \sh{k} C'L\sp{0}  \ar[ur]  \ar[rr] && \cdots \sh{k} C'L\sp{n-k-1}) \ar[ur] \ar@/_2ex/[urrr]
}
}
\normalsize{}
\noindent Here \w{\hdtw\sb{i}} is just \w{\hdtw} for the $i$-hybridization \w[,]{\Hyb{i}\fM{n}} 
and we have omitted the factorization depicted on the left in \wref[,]{eqhybonem} collapsing the two copies into one.

The general \wwb{k+1}derived diagram  \w{\Der{k+1} \fM{n}= (\Hyb{k}\fM{n})'} 
for \w{k<n} is given by
\footnotesize{
\myydiag[\label{eqderkm}]{
&& &   &\sh{k+1}V\ar[r]^-{\psi} & 0   \ar[r] \ar[rdd] & \cdots \ar[r]  & 0  \ar[dr] \\ 
V \ar[r]^-{\hdtw_1} &\sh{} V \cdots \sh{k-1}V \ar[r]^-{\hdtw\sb{k}} & \sh{k}V  \ar[rru]^-{\hdtw} \ar[rrd]_-{\hdtw}  & & &&  &    &   H\sb{k+1} K(V,n)~. \\ 
 && & &\sh{k+1}V\ar[r]_-{\psi'}& 0    \ar[r] \ar[ruu] & \cdots \ar[r]  & 0  \ar[ur]
}
}
\normalsize{}

The only derived diagram with non-trivial values for its differentials is \w[:]{\Der{n} \fM{n} } 
\myydiag[\label{eqdermm}]{
V \ar[rr]\sp{\hdtw_1} && \Sigma\sp{1} V \ar[r] & \cdots\ar[r] & \Sigma\sp{n-1} V \ar[rr]\sp{\hdtw_{n}} && \Sigma\sp{n}V    \ar[rr]\sp{\psi} && \Sigma\sp{n} V~. 
}
\noindent In this case we are recording the existence of a \ww{d\sp{n+1}}-differential (see \S \ref{sthreediff}):
the formal part of this differential is given by the composite
\w[,]{\hdtw_n \circ \cdots \circ  \hdtw_1} and its value is determined by 
post-composition with $\psi$ (which we assume to be an isomorphism in the ``standard version'' of \w[).]{\fM{n}}
\end{mysubsection}

%
%c5      Spectral sequences
%
\sect{Spectral sequences}\label{css}
As our terminology suggests, the approach described here may be helpful in understanding spectral sequences of chain complexes. 
For simplicity, we concentrate mainly on the simplest case:  that of a double complex (see \cite[III, \S 13]{BTuD}).

\begin{mysubsection}{Double complexes and cubical diagrams}
\label{sdccd}
A double complex \w{\Css} can be thought of as a linear diagram
\begin{myeq}\label{eqlindiag}
\dotsc C\sb{n}\xra{\pd{n}}C\sb{n-1}\xra{\pd{n-1}}\dotsc C\sb{1}\xra{\pd{1}}C\sb{0}\dotsc
\end{myeq}
\noindent in \w[,]{\Ch} satisfying 
\begin{myeq}\label{eqdsqzero}
\pd{n-1}\circ\pd{n}=0\hsm \text{for each}\hsm  n~.
\end{myeq}

To simplify matters, we consider only a finite segment of the (possibly bi-infinite)
double complex \wref{eqlindiag} \wwh say, one of length $n$, from \w{C\sb{n}} to 
\w{C\sb{0}} (inclusive).  As we shall see below, this suffices to determine the differentials up to order $n$ originating in \w[;]{C\sb{n}} if we wish to determine 
\w{E\sp{r}\sb{n\ast}} for \w[,]{r\leq n+1} we must double the length of the segment, and start it at \w[.]{C\sb{2n}} 

However, according to the results of Section \ref{cdd}, the diagram \wref{eqlindiag} as such is completely determined by its primary structure (that is,
the diagram \w{\Hs C\sb{n}\to\Hs C\sb{n-1} \to \dotsc \to \Hs C\sb{0}} in \w[),]{\grV}
since its indexing category $\sI$ is linear. The reason is that we have not recorded
\wref{eqdsqzero} in diagrammatic form: in order to do so, we must use
cubical indexing diagrams.

Given such a segment \w[,]{C\sb{n}\to C\sb{n-1}\to\dotsc C\sb{1}\to C\sb{0}} we define
the \emph{associated cube diagram} to be \w{\fX:\sI\sp{n}\to\Ch} defined by \w{\fX(J\sb{A\sb{i}})=C\sb{i}} for \w{J\sb{A\sb{i}}=(1\dotsc 10\dotsc 0)} consisting of $i$ digits $1$ followed by \w{n-i} zeros, and  \w{\fX(J\sb{A})=0} for all other $A$ 
(see \S \ref{scube}). The unique morphism in \w{\sI\sp{n}} from \w{J\sb{A\sb{i}}} to \w{J\sb{A\sb{i-1}}} is sent to \w[.]{\pd{i}:C\sb{i}\to C\sb{i-1}} 
Evidently these diagrams commute strictly, given \wref[.]{eqdsqzero}
\end{mysubsection}

\begin{example}\label{egss}
In \wref{eqthreecube} we depict the associated cube diagrams for \w[.]{n=2,3} 

\myqdiag[\label{eqthreecube}]{
%
%                  Top left edge
%
\stackrel{\mbox{$C\sb{2}$}}{\mbox{\scriptsize{(11)}}}
\ar@<2ex>[rrr] \ar[dd]\sp(0.4){\pd{2}}
&&&
\stackrel{\mbox{$0$}}{\mbox{\scriptsize{(01)}}}  \ar[dd]
&&&&&
%
%                  Top right square
%
\stackrel{\mbox{$C\sb{3}$}}{\mbox{\scriptsize{(111)}}}
\ar@<2ex>[rrr] \ar@<1ex>[rrd]\ar[dd]\sp(0.4){\pd{3}}
&&&
\stackrel{\mbox{$0$}}{\mbox{\scriptsize{(101)}}}  \ar[dd]|(.34){\hole}  
\ar@<1ex>[rrd]
&&&\\
%
%                  Front row
%
&&&&&&&& && \stackrel{\mbox{$0$}}{\mbox{\scriptsize{(011)}}} 
\ar@<2ex>[rrr] \ar[dd]
&&& \stackrel{\mbox{$0$}}{\mbox{\scriptsize{(001)}}} \ar[dd]
\\
%
%                  bottom left edge
%
\stackrel{\mbox{$C\sb{1}$}}{\mbox{\scriptsize{(10)}}}
\ar@<2ex>[rrr]\sp{\pd{1}}
&&&
\stackrel{\mbox{$C\sb{0}$}}{\mbox{\scriptsize{(00)}}}
&&&&&
%
%              Bottom square
%
\stackrel{\mbox{$C\sb{2}$}}{\mbox{\scriptsize{(110)}}}
\ar@<2ex>[rrr]|(.59){\hole}\sp(0.4){\pd{2}}  \ar@<1ex>[rrd]
&&&
\stackrel{\mbox{$C\sb{1}$}}{\mbox{\scriptsize{(100)}}} 
  \ar@<1ex>[rrd]\sp{\pd{1}} \\
%
%            Front row
%
&&&&&&&& && \stackrel{\mbox{$0$}}{\mbox{\scriptsize{(010)}}}
\ar@<2ex>[rrr]
&&&
\stackrel{\mbox{$C\sb{0}$}}{\mbox{\scriptsize{(000)}}}
}

\begin{enumerate}
\renewcommand{\labelenumi}{\Roman{enumi}.~}
\item The $2$-dimensional cube:

For $\fX$ the left diagram in \wref[,]{eqthreecube} the derived diagram \w{\Der{1}\fX}
is depicted in \wref[,]{eqcubea} in which the outer square shows the primary information
in degree $0$ (omitting the same information in degree $1$ in the interests of readability).
The inner diagonal part of the diagram shows the formal differential 
\w[,]{\hdtw=\hdt{01,10}} defined only on the common kernel of the top and left maps,
followed by its value $\psi$ into \w[.]{H\sb{1}C\sb{0}} The composite of these two is in fact the usual differential \w[,]{d\sp{2}:E\sp{2}\sb{2,0}\to E\sp{2}\sb{0,1}} 
since in this simple example \w{E\sp{2}\sb{2,0}=\Ker(H\sb{0}(\pd{2}))} (see \S \ref{sdss}).

\myudiag[\label{eqcubea}]{
H\sb{0}C\sb{2} \ar[dddd]\sb{H\sb{0}(\pd{2})} \ar[rrrr] &&&& 0 \ar[dddd] \\
& \Ker(H\sb{0}(\pd{2})) \ar@{_{(}->}[lu]^{i} \ar[rd]\sp{\hdtw}\sb{\cong} &\\
&& \Sigma\Ker(H\sb{0}(\pd{2})) \ar[rd]\sp(0.6){\psi} &\\ 
&&&H\sb{1}C\sb{0} \ar@{}[rd]\sp{\oplus} &\\
H\sb{0}C\sb{1} \ar[rrrr]\sb{H\sb{0}(\pd{1})} &&&& H\sb{0}C\sb{0}
}

In this small example we could take the direct sum of diagrams as in 
\wref{eqcubea} in all degrees, thus reproducing all the secondary information in one amalgamated derived diagram (see Remark \ref{roas}).

\item The $3$-dimensional cube:

The full derived diagrams for $\fX$ in the right cube of \wref{eqthreecube} is complicated to describe, but on the assumption that all homology groups vanish except \w[,]{H\sb{0}C\sb{3}} \w[,]{H\sb{1}C\sb{1}} and \w[,]{H\sb{2}C\sb{0}} so necessarily 
\w[.]{H\sb{0}(\pd{3})=0}  In this case the relevant part of \w{\Der{1}\fX} 
is shown in \wref[:]{eqcubeb} 

\myvdiag[\label{eqcubeb}]{
&& \Sigma V \ar[r]\sp{\psi\sb{011,101}}& 0\\
V=H\sb{0}C\sb{3} \ar[rru]\sp{\hdtw}\sb{\cong}
\ar[rr]\sp(0.7){\hdtw}\sb(0.7){\cong}
\ar[rrd]\sp{\hdtw}\sb{\cong} && \Sigma V \ar[r]\sp{\psi\sb{110,101}}& H\sb{1}C\sb{1} \\
&& \Sigma V \ar[r]\sp{\psi\sb{110,011}}& 0
}
%(
\noindent In principle, we have a similar diagram for the bottom, front and right facets, which we may disregard for our purposes.

If we further assume that \w{\psi\sb{110,101}=0} (so \w{d\sp{2}:E\sp{2}\sb{3,0}\to E\sp{2}\sb{1,1}} vanishes), we are in the case of \wref[,]{eqexpcubec} and can proceed to 
the next derived diagram as explained there.
\end{enumerate}
\end{example}

\begin{mysubsection}{Differentials in spectral sequences}
\label{sdss}
The justification for our claim that the associated cube diagrams are the correct way to represent
the spectral sequence of a double complex is the fact that they can be used to derive the
successive differentials in this spectral sequence.

Following \cite{BMeadA}, we may think of the differentials in the spectral sequence of a filtered chain complex (or more generally, a simplicial object in a pointed model or $\infty$-category) in terms of values of higher Toda brackets, in the sense of \cite{BBSenT} or \cite{KharoH}.
More precisely, rather than describing the differential as a linear map
\w[,]{d\sp{r}:E\sp{r}\sb{\ast\ast}\to E\sp{r}\sb{\ast\ast}} we think of it as a ``relation''
in the sense of \cite[\S 3]{BKaSQ}: that is, we  take a representative in 
\w{\alpha\in E\sp{1}\sb{pq}} of the class \w[,]{\lra{\alpha}\in E\sp{r}\sb{pq}} and look for any representative \w{\beta\in E\sp{1}\sb{\ast\ast}} of \w{d\sp{r}\lra{\alpha}} (with
the understanding that we have used all lower choices to determine the subquotient
\w{E\sp{r}\sb{pq}} of \w[).]{E\sp{1}\sb{pq}}

It turns out to be more convenient to use the fibrant, rather than cofibrant, replacement $\fX$ of the associated cube diagrams \wref{eqthreecube} in the injective model category structure on \w{\Ch\sp{\sI\sp{n}}} (see Remark \ref{rfibrant} above). We can then form the \emph{extended cube diagram} $\tX$, indexed by an $n$-dimensional cube with sides of length $3$, obtained by taking the (homotopy) fibers of all maps (see \wref[).]{eqtodasq}

\begin{remark}\label{rfibers}
Formally, this means that we add arrows from the zero object to all targets of fibrations, and then take the corresponding (homotopy) limits. We do not display these explicitly in the extended cube diagrams for reasons of clarity only.
\end{remark}

Finally, using \w{A=\Sigma\sp{p}\bF=K(\bF,p)} to co-represent 
\w[,]{\alpha\in E\sp{1}\sb{pq}=H\sb{p}C\sb{q}} in \w{\ho\Ch} we define the \emph{enhanced cube diagram} for $\fX$ to be \w{\tX\sp{A}} (using the internal \w{\Hom} in \w[).]{\Ch}

When \w{n=2} (the first interesting case), the extended diagram for
\w{C\sb{2}\xra{\pd{2}} C\sb{1}\xra{\pd{1}}C\sb{0}} is
\mydiagram[\label{eqtodasq}]{
\stackrel{\mbox{$Z\sb{2}=\Fib{g\sb{2}}$}}{\mbox{\scriptsize{(22)}}}
\ar@<2ex>@{>->}[rr]\sp{i\sb{2}}  \ar@{>->}[d]\sb(0.4){v\sb{2}}
&&
\stackrel{\mbox{$\Fib{\tpd{2}}$}}{\mbox{\scriptsize{(12)}}}
\ar@{>->}[d] \ar@<2ex>@{->>}[rr]\sp{k\sb{2}}
&&
\stackrel{\mbox{$\Omega C\sb{0}$}}{\mbox{\scriptsize{(02)}}} \ar@{>->}[d] \\
\stackrel{\mbox{$C\sb{2}$}}{\mbox{\scriptsize{(21)}}}
\ar@{->>}[d]\sb(0.3){g\sb{2}} \ar@<2ex>@{>->}[rr]\sp{j}\sb{\simeq}
\ar@<2ex>[rrd]\sp{\pd{2}}
&&
\stackrel{\mbox{$C'\sb{2}$}}{\mbox{\scriptsize{(11)}}}
\ar@{->>}[d]\sp(0.4){\tpd{2}} \ar@<2ex>@{->>}[rr]\sp{G}
&&
\stackrel{\mbox{$PC\sb{0}$}}{\mbox{\scriptsize{(01)}}} \ar@{->>}[d]\sp(0.4){p} \\
\stackrel{\mbox{$Z\sb{1}=\Fib{\pd{1}}$}}{\mbox{\scriptsize{(20)}}}
\ar@<2ex>@{>->}[rr]\sb{v\sb{1}}
&&
\stackrel{\mbox{$C\sb{1}$}}{\mbox{\scriptsize{(10)}}} \ar@<2ex>@{->>}[rr]\sb{\pd{1}}
&&
\stackrel{\mbox{$C\sb{0}$}}{\mbox{\scriptsize{(00)}}}
}
\noindent The lower right square is the fibrant version of the associated cubical diagram 
for \w[,]{\Css} with the nullhomotopy $G$ as witness for the fact that
\w{\pd{1}\circ\tpd{2}\sim 0} (the composite is no longer strictly zero in the fibrant replacement).
See \cite[(0.3)]{BSenH} and compare \cite[(1.17)]{BBSenT}.

Now an element \w{\varphi\in H\sb{i}C\sb{2}} is represented by a map \w[.]{f:K(\bF,i)\to C\sb{2}}
In order to carry out a diagram chase on $\varphi$ mapping into \wref[,]{eqtodasq} it is convenient to work 
with a point-set model of the mapping object from \w{A:=K(\bF,i)} into \wref[,]{eqtodasq} as follows:
\myvdiag[\label{eqtwotodasq}]{
\stackrel{\mbox{$\stackrel{\exists?\,[h']}{\vin}$}}
           {\stackrel{\mbox{$Z\sb{2}\sp{A}$}}{\mbox{\scriptsize{(22)}}}}
\ar@<6.5ex>@{|->}[rr]
\ar@<2ex>@{>->}[rr]\sp{i\sb{2}}  \ar@{>->}[d]\sb(0.4){v\sb{2}\sp{A}}
&&
\stackrel{\mbox{\hspace*{-4mm}$\stackrel{[h]}{\vin}$}}{\stackrel{\mbox{$\Fib{\pd{2}}\sp{A}$}}
  {\mbox{\hspace*{5mm}\scriptsize{(12)}}}}
\ar@<-2ex>@{|->}[d]\ar@<2ex>@{>->}[d] \ar@<2ex>@{->>}[rr]\sp{k\sb{2}\sp{A}}
\ar@<6.5ex>@{|->}[rr]
&&
\stackrel{\mbox{$\hspace*{-4mm}\stackrel{\tau\stackrel{?}{\sim}\ast}{\vin}$}}{\stackrel{\mbox{$\Omega C\sb{0}\sp{A}$}}{\mbox{\scriptsize{(02)}}}} \ar@{>->}[d] \\
\stackrel{\mbox{$C\sb{2}\sp{A}$}}{\mbox{\scriptsize{(21)}}}
\ar@{->>}[d]\sb(0.3){g\sb{2}\sp{A}} \ar@<2ex>[rr]\sp{\simeq}
&&
\stackrel{\mbox{$[f]\in C\sb{2}\sp{A}$}}
  {\hspace*{5mm}\mbox{\scriptsize{(11)}}}
  \ar@<-2ex>@{|->}[d]\ar@<2ex>@{->>}[d]\sp(0.3){\pd{2}\sp{A}} \ar@<2ex>[rr]
&&
\stackrel{\mbox{$\ast$}}{\mbox{\scriptsize{(01)}}} \ar@{->>}[d]\sp(0.3){p\sp{A}} \\
\stackrel{\mbox{$Z\sb{1}\sp{A}$}}{\mbox{\scriptsize{(20)}}}
\ar@<2ex>@{>->}[rr]\sb{v\sb{1}}
&&
\stackrel{\mbox{$0\in C\sb{1}\sp{A}$}}{\mbox{\scriptsize{(10)}}}
\ar@<2ex>@{->>}[rr]\sb{\pd{1}\sp{A}}
&&
\stackrel{\mbox{$C\sb{0}\sp{A}$}}{\mbox{\scriptsize{(00)}}}.
}
\noindent Note that the columns and rows of \wref{eqtwotodasq} are still (homotopy) fibration
sequences. We see that \w{\varphi=[f]} is represented in the central slot, and 
the fact that it is a cycle under \w{d\sp{1}:H\sb{i}C\sb{2}\sp{A}\to H\sb{i}C\sb{1}\sp{A}} means that it maps 
to zero in the central column, thus lifting to \w{h:A\to\Fib{\pd{2}}} \wwh with the choice of lift representing the indeterminacy in choosing the representative for \w{d\sp{2}(\varphi)} in
\w[,]{E\sp{1}\sb{0,i+1}} namely, \w[.]{k\sb{2}\sp{A}\circ h} If this vanishes, \w{[h]} lifts to 
\w{[h']} in the upper left  corner \wh which in this simple example represents the \w{E\sp{\infty}\sb{2,i}} term of the spectral sequence.
\end{mysubsection}

\begin{mysubsection}{Filtered complexes}
\label{sfcomp}
More generally, the spectral sequence associated to a filtered chain complex
\mysdiag[\label{eqfconp}]{
\dotsc & F\sb{0}~~  \ar@{^{(}->}[r]^{\iota\sb{0}} &  ~~F\sb{1}~~  \ar@{^{(}->}[r]^{\iota\sb{1}} &  ~~F\sb{2}~~ & \dotsc &  ~~F\sb{n}~~\ar@{^{(}->}[r]^{\iota\sb{n}} &  ~~F\sb{n+1}  & \dotsc
}
\noindent is actually that associated to a linear diagram of the form \wref[,]{eqlindiag} in which 
\w{C\sb{n}=\Sigma\sp{-n}(F\sb{n}/F\sb{n-1})} for each $n$.
The structure map \w{\partial\sb{n}:C\sb{n}\to C\sb{n-1}} for this diagram
is induced (under $n$-fold desuspension) by the connecting homomorphism \w{\overline{\partial}} for the cofiber sequence
\mysdiag[\label{eqcofseq}]{
F\sb{n-1}~~\ar@{^{(}->}[rr]^{\iota\sb{n-1}} &&  F\sb{n}  \ar@{->>}[rr]^-{q\sb{n}} && 
F\sb{n}/F\sb{n-1}~=~\Sigma\sp{n}C\sb{n} \ar[rr]\sp-{\overline{\partial}} && \Sigma F\sb{n-1}~,
}
\noindent followed by the (suspended)  quotient map 
\w[.]{\Sigma q\sb{n-1}:\Sigma F\sb{n-1}\to\Sigma\sp{n} C\sb{n-1}} 

As in \S \ref{sdccd}, to encode the correct homotopy type of the filtered complex, we must expand each finite $n$-segment of \wref{eqfconp} into the associated $n$-cube diagram, as shown in \wref{eqtwothreecube} for \w[.]{n=2,3} Note that we use here the indexing needed for the cofibration version, which is more
appropriate for our description of the spectral sequence, using cofibers.

\myqdiag[\label{eqtwothreecube}]{
%
%                  Top left edge
%
\stackrel{\mbox{$F\sb{0}$}}{\mbox{\scriptsize{(00)}}}
\ar@<2ex>[rrr]\sp{\iota\sb{1}\circ\iota\sb{0}} \ar[dd]\sp(0.4){\iota\sb{0}}
&&&
\stackrel{\mbox{$F\sb{2}$}}{\mbox{\scriptsize{(10)}}}  \ar[dd]\sb{=}
&&&&&
%
%                  Top right square
%
\stackrel{\mbox{$F\sb{0}$}}{\mbox{\scriptsize{(000)}}}
\ar@<2ex>[rrr]\sp{\iota\sb{1}\circ\iota\sb{0}}  \ar@<1ex>[rrd]\sp{\iota\sb{2}\circ\iota\sb{1}\circ\iota\sb{0}}  \ar[dd]\sp(0.4){\iota\sb{0}}
&&&
\stackrel{\mbox{$F\sb{2}$}}{\mbox{\scriptsize{(010)}}}  \ar[dd]|(.34){\hole}^{=}  
\ar@<1ex>[rrd]\sp{\iota\sb{2}} 
&&&\\
%
%                  Front row
%
&&&&&&&& && \stackrel{\mbox{$F\sb{3}$}}{\mbox{\scriptsize{(100)}}} 
\ar@<2ex>[rrr]^(0.3){=} \ar[dd]
&&& \stackrel{\mbox{$F\sb{3}$}}{\mbox{\scriptsize{(110)}}} \ar[dd]^{=}
\\
%
%                  bottom left edge
%
\stackrel{\mbox{$F\sb{1}$}}{\mbox{\scriptsize{(01)}}}
\ar@<2ex>[rrr]\sp{\iota\sb{1}}
&&&
\stackrel{\mbox{$F\sb{2}$}}{\mbox{\scriptsize{(11)}}}
&&&&&
%
%              Bottom square
%
\stackrel{\mbox{$F\sb{1}$}}{\mbox{\scriptsize{(001)}}}
\ar@<2ex>[rrr]|(.59){\hole}\sp(0.4){\iota\sb{1}}  \ar@<1ex>[rrd]\sp{\iota\sb{2}\circ\iota\sb{1}} 
&&&
\stackrel{\mbox{$F\sb{2}$}}{\mbox{\scriptsize{(011)}}} 
  \ar@<1ex>[rrd]\sp{\iota\sb{2}} \\
%
%            Front row
%
&&&&&&&& && \stackrel{\mbox{$F\sb{3}$}}{\mbox{\scriptsize{(101)}}}
\ar@<2ex>[rrr]^{=}
&&&
\stackrel{\mbox{$F\sb{3}$}}{\mbox{\scriptsize{(111)}}}
}

The extended diagram for the right-hand associated cube in \wref{eqtwothreecube} contains the data needed to recover the left-hand square of \wref{eqthreecube} (which is the (unextended) associated cube diagram of the appropriate segment of \wref[).]{eqlindiag} 
Note that in order to extract the data needed for an $n$ segment of the corresponding \wref[,]{eqlindiag} 
(and the associated $n$-dimensional cube), we need an \wwb{n+1}segment of \wref[,]{eqfconp} and thus the associated
\wwb{n+1}cube diagram. For example, the square \wref[,]{eqtwoextend} associated to the left square in 
\wref[,]{eqtwothreecube} only yields the single arrow \w[.]{\partial\sb{2}:C\sb{2}\to C\sb{1}}
Note also that in order to conform to the usual conventions for filtered complexes, we depict here the version with all rows and columns (homotopy) cofibration sequences.

\mydiagram[\label{eqtwoextend}]{
\stackrel{\mbox{$F\sb{0}$}}{\mbox{\scriptsize{(00)}}}
\ar@<2ex>@{>->}[rr]\sp{\iota\sb{1}\circ\iota\sb{0}}  
\ar@{>->}[d]\sb(0.4){\iota\sb{0}}
&&
\stackrel{\mbox{$F\sb{2}$}}{\mbox{\scriptsize{(10)}}}
\ar@{>->}[d]\sp{=} \ar@<2ex>@{->>}[rr]\sp{k\sb{2}}
&&
\stackrel{\mbox{$F\sb{2}/F\sb{0}$}}{\mbox{\scriptsize{(20)}}} \ar@{>->}[d] \\
\stackrel{\mbox{$F\sb{1}$}}{\mbox{\scriptsize{(01)}}}
\ar@{->>}[d]\sb(0.3){q\sb{1}} \ar@<2ex>@{>->}[rr]\sp{\iota\sb{1}}
&&
\stackrel{\mbox{$F\sb{2}$}}{\mbox{\scriptsize{(11)}}}
\ar@{->>}[d] \ar@<2ex>@{->>}[rr]\sp{q\sb{2}}
&&
\stackrel{\mbox{$\Sigma\sp{2}C\sb{2}=F\sb{2}/F\sb{1}$}}{\mbox{\scriptsize{(21)}}} 
\ar@{->>}[d]\sp(0.4){\Sigma\sp{2}\partial\sb{2}} \\
\stackrel{\mbox{$\Sigma C\sb{1}=F\sb{1}/F\sb{0}$}}{\mbox{\scriptsize{(02)}}}
\ar@<2ex>@{>->}[rr]
&&
\stackrel{\mbox{$0$}}{\mbox{\scriptsize{(12)}}} \ar@<2ex>@{->>}[rr]
&&
\stackrel{\mbox{$\Sigma\sp{2}C\sb{1}$}}{\mbox{\scriptsize{(22)}}}
}
\end{mysubsection}

%
%c6 
%

\sect{Further directions}
\label{cfd}
So far we have discussed a very specific type of diagrams - those indexed by finite partially ordered sets. However, this case is the most basic example, and should offer a useful guidepost for more general cases.

\begin{mysubsection}{Parallel arrows}
\label{spa}
If we assume that the indexing category $\sI$ is still directed (and finite), but allows parallel arrows, the diagrams become more complicated in terms of the linear algebra (or representation theory) needed to understand the primary information (that is, the diagram \w[).]{H\sb{\ast}\fX:\sI\to\grV}

However, the changes needed to deal with the \emph{higher order} information, which are our main concern here, are minor. This is because if $\fX$ is cofibrant, then given two (or more) arrows \w{f,g:\alpha\to\beta} in $\sI$, on \w[,]{V:=\Ker(f\sb{\ast})\cap\Ker(g\sb{\ast})\subseteq H\sb{k}\fX(\alpha)} say, we will have disjoint copies of the cone \w{CV} in \w[.]{\fX(\beta)} 
Thus the (homotopy) colimit of the diagram restricted to $V$ will be the pushout \w[,]{\Sigma V}
as in \S \ref{spairs}. In effect, we may replace the indexing category $\sI$ by an expanded version
\w{\sI'} in which, in addition to $\beta$ (serving as a target of the non-zero part of the maps 
\w{f\sb{\ast}} and \w[),]{g\sb{\ast}} we have two new copies \w{\beta\sb{f}} and \w[,]{\beta\sb{g}}
at which we place the respective cones \w{CV} to obtain \w{\fX':\sI'\to\Ch} (as in the construction of the expanded derived diagram in \S \ref{dedd}). Note that the collapsing functor \w{c:\sI'\to\sI} allows us to map \w{\fX'} to \w[,]{c\sp{\ast}\fX} and thus compare the two.
\end{mysubsection}

\begin{mysubsection}{Infinite diagrams}
\label{sidiag}
Infinite directed diagrams may in principle be treated as the colimits of their finite segments. However, the case when $\sI$ is not directed \wh even in the simple case of a single non-trivial self map \wh can be much more complicated. One may try to approximate a diagram \w{\fX:\sI\to\Ch} by finite segments \wh e.g., consider only a finite number of iterations of a self map, thought of as a diagram indexed by a finite linearly ordered set. However, as the case of homotopy actions of the cyclic group \w{C\sb{2}} shows, these finite approximations may not provide the full picture.
\end{mysubsection}

\end{document}